\documentclass[12pt]{article}
\usepackage[latin1]{inputenc}
\usepackage[english]{babel}

\usepackage{amsmath,amssymb,amsfonts, amsthm}

\usepackage{hyperref}

\usepackage{float}

\usepackage{xfrac}

\newif\ifpdf
\ifx\pdfoutput\undefined
   \pdffalse
\else
   \pdfoutput=1
   \pdftrue
\fi
\ifpdf
   \usepackage{graphicx}
   \usepackage{epstopdf}
\else
   \usepackage{graphicx}
\fi

\newcommand{\R}{\mathbb R}
\newcommand{\C}{\mathbb C}
\DeclareMathOperator{\ind}{Ind}
\newcommand{\D}{\Delta}
\newcommand{\G}{\Gamma}
\DeclareMathOperator{\di}{d}
\DeclareMathOperator{\dm}{dm}
\newcommand{\dr}{\dm_{Rect}}
\newcommand{\la}{\lambda}

\DeclareMathOperator{\IPSR}{IPSR}
\DeclareMathOperator{\RDP}{RDP}
\newcommand{\ld}{\lg_2}

\newtheorem{propositionn}{Proposition}

\newtheorem{Corollaryn}{Corollary}

\newtheorem{lemman}{Lemma}
\newtheorem{Theorem}{Theorem}

\everymath{\displaystyle}

\begin{document}
%
%
%
%
%
%
%
%
%
%
%
%
%
%
%
%
%
%

\newpage

\begin{center}
{\bf A Geometrical Root Finding Method for Polynomials, with Complexity Analysis}

\textit{Juan-Luis García Zapata}, Departamento de Matemáticas

 \textit{Juan Carlos Díaz Martín}, Departamento de  Tecnología de los
Computadores y las comunicaciones 

Escuela Politécnica, Universidad de Extremadura

10003, Cáceres (Spain)
\end{center}

\begin{abstract}

The usual methods for root finding of polynomials are based on the iteration of a numerical formula for improvement of successive estimations. The unpredictable nature of the iterations prevents to search roots inside a pre-specified region of complex plane. This work describes a root finding method that overcomes this disadvantage. It is based on the winding number of plane curves, a geometric construct that is related to the number of zeros of each polynomial. The method can be restrained to search inside a pre-specified region. It can be easily parallelized.
 Its  correctness is formally proven, and its computational complexity is also analyzed.
\end{abstract}


\section{Introduction: Root Finding as a Geometric Problem}
In this section we first compare the iterative methods of root finding with the geometric methods, then we display a graphic example of the disadvantages of the former ones, and finally expose the practical motivation of the geometric method shown in this paper.

The methods to find polynomial roots can be roughly classified as {\em iterative} or as {\em geometric}. The iterative methods (IM) are based on a sequence of error estimation and correction, which, in well-defined conditions, leads to a complex point as close to a root as needed. The methods of this type are fast (better than linear convergence, in general) and its analysis (that is, the  proof of its correction, and the determination of the resources needed) relies on relatively standard numerical techniques (\cite{ralston}). They are the common approach to the polynomial root finding problem, named after Newton, Müller, Legendre, etc. For example, the procedure built in mathematical packages (as LAPACK \cite{lapack} or MATLAB \cite{dongarra}) is an IM which is based on the eigenvectors of the companion matrix of the polynomial (see \cite{fortune}). An IM usually converges rapidly for most
polynomials of moderate degree (less than about 50 \cite{pan}). However, even for
these degrees, the IM are inadequate for certain classes
of polynomials (as those with clusters or with multiple roots, depending
on the method), or for specific polynomials that show ill-conditioning
(like the Wilkinson example \cite{wilkinson}). To cope with these issues, it is necessary to
apply heuristic rules in case of difficulties (to choose another initial
approximations, or change the method applied, or use
high precision arithmetic \cite{pan}). These matters hinder the algorithmic complexity analysis of IM. For example, for the Newton method there is not such complexity analysis (see \cite{traub}, \cite{forster}). However, in practice most applications of root finding rely on IM, with an proper set of heuristics. 

Another disadvantage of IM is that the position of the found root is not related to the initial estimation  \cite{blum}. This can be pictured in a figure of basins of attraction. The {\em basin of attraction} of a root (for the Newton method) is the set of complex values $z$ such that the method converges to that root, if we start with initial value $z$. It is said that a root {\em attracts} the points that converge to it. The basins of two distinct roots are disjoint subsets of the plane. We depict the basins of the roots of a 31-degree polynomial in figure \ref{basin}. It is the plane region $|\text{Re}(z)| < 1.2,  |\text{Im}(z)| < 1.2$. The roots are marked with stars. The Newton method is applied one time for every pixel of the figure. The initial root estimation is the complex value corresponding to the position of this pixel. Each pixel is colored with a gray level, according to the root to which it converges. Hence each basin is colored with a different gray level.

\begin{figure}[h!]
  \centering \includegraphics[width=\textwidth, bb= 10 10 200 165]{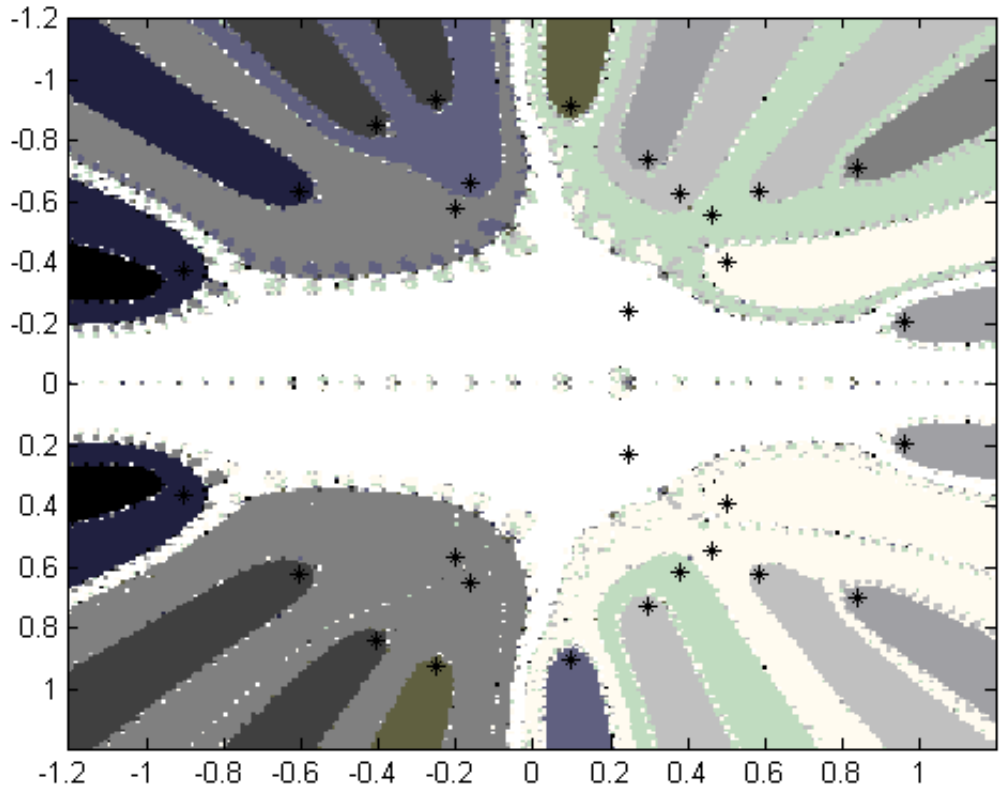}
  \parbox{.9\textwidth}{\caption{Basins of attraction for the Newton method applied to a 31 degree polynomial. The stars mark the roots, and the gray levels its basins.}}
  \label{basin}
\end{figure}

In the figure, the boundaries between basins show points whose color is different from that of its neighborhoods. One such point converges to a root located far apart from the roots corresponding to nearby points. Then the root to which a point converges is unpredictable. This {\em non-local convergence} is characteristic of IM. It is a phenomenon of sensitive dependence of initial conditions, and it cannot be avoided with any level of precision \cite{kalantari}.  

In this setting, to find for instance the roots near the unit circumference with common IM, it is necessary to find all roots, and pick up those of interest. To avoid the repeated finding of the same root,  each found root is removed from the polynomial by dividing it by a linear factor (a so called deflation step). The deflation steps require high precision arithmetic, which is computationally intensive \cite{ralston}.

In contrast with IM, others methods of root finding are based on geometric notions about the distribution of roots in the complex plane, as its relative position, or if they are contained inside a given circle, etc. Traditionally, geometric methods (GM) are applied in a first stage of root finding, as ``rules of thumb'' to separate intervals on the real line (or to separate regions on the complex plane) each containing a desired root. On a second stage, an iterative method (that now in each separated interval supposedly converges) approximates the root in every interval. Tools used for root separation in the real line include the criterion of Cauchy for root signs, the bisection method (based on Bolzano's theorem), or the Sturm successions \cite{ralston}. Examples of root separation tools in the complex plane are Lehmer-Schur \cite{lehmer}, Weyl \cite{henrici,yakoubsohn}, or the process of Graeffe \cite{bini}. The GM are more difficult to implement than iterative methods, requiring data types for geometric objects and complex variables, and tree-search or backtracking procedures for control flow.
Notwithstanding, the methods based on geometrical relations are  valid for all polynomials equally. This uniformity allows an analysis of the complexity of such methods. 
The theoretical studies of complexity have been a driving force in the development of GM (\cite{renegar}). Suppose that we require the roots of a polynomial with up to $b$ bits, that is, with a precision of $2^{-b}$. The number of multiplications needed to extract all the roots of a $n$-degree polynomial, with this precision, is $O(n^2 \log n \log b)$ using the GM of \cite{pan}, or $O(n \log n \log b)$ in \cite{neff-reif}.

Contrarily to what happens with IM, the use of deflation steps (and hence of high precision) is not necessary in GM. Besides, if the area of interest is relatively reduced, and contains few roots, a GM avoids spending computations in not desired roots. Many practical applications need methods that allow to focus the root finding on a specified complex plane region (see \cite{sitton}, \cite{vandooren} and references therein). This ability produces a computational saving with respect to IM.  

In our work \cite{garciaLPC} we face a situation where a geometric method is better than an iterative one. This advantage mainly comes from that in many cases there is not need to find all and each root of a polynomial, but only those ones satisfying certain conditions. In signal processing applications, a polynomial represents an autoregressive model (Linear Predictive Coding, or LPC) of the signal \cite{oppenheim}. The roots of this polynomial are inside the unit circle, by the stability of the model, and the roots that are closer to the circumference are related to the main frequency components of the signal. For instance the polynomial of figure \ref{basin} comes from LPC of a voice signal. We are interested in finding the complex roots of LPC polynomials that are placed near the unit circumference, up to a distance that depends on the degree of the modeling. This leads us to consider GM for this problem. In signal processing there are several IM  to find the roots of greater module, as Graeffe or Bernoulli \cite{ralston}. However, the high degree of the polynomials of interest for us, and the intricate case analysis required by these methods, are two features that obstruct its application, and  the practical methods of voice processing must use spectral analysis techniques. 

In general, the GM have a common pattern. Let us consider, for expository purpose, the bisection method for zero finding of real continuous functions. It is based on the Bolzano's theorem. This theorem says that if a continuous function $f:[a,b]\rightarrow \R$ in a interval $[a,b]$ of the real straight takes opposed signs at its ends (that is,  $f(a)\cdot f(b)<0$), then it has at least a root in the interval. The bisection method consists of using this fact recursively to build a succession of nested intervals (and hence progressively smaller) containing a real root. These are the prototype components of a GM: an {\em inclusion test}, to decide whether there is any root in a piece (either of the line or of the plane); and a {\em recursive procedure} to get progressively smaller  localizations of desired roots.

We discuss two examples for comparison: Lehmer-Schur and  Weyl methods. As the GM are based on the delimitation of regions that can contain roots, they are often called bracketing methods, by analogy with the bisection ones. The Lehmer-Schur method \cite{lehmer} is a bi-dimensional bracketing method based on an inclusion test by Schur and Cohn, which says whether a polynomial has roots inside a circle. The area of interest is covered by circles, and this criterion is applied to each one. Then, the circles that contain a root are covered by circles of lesser radius, and so on, until the required precision is reached. The recursive procedure entails certain inefficiency because it implies overlapping and repeating finding of the same roots. The bracketing method of Weyl \cite{henrici} is based on the decomposition in squares of the domain of interest. Each square is subdivided in four sub-squares, if it contains some root, and so recursively. This leads to a {\em quadtree procedure},  a search tree in which each node is a square that is linked with its four sub-squares. Weyl and others \cite{pan} proposed several inclusion tests  to see if there are roots inside a square. So, ``method of Weyl'' is a generic name of several GM, with a common recursive procedure (the quadtree), but with different inclusion tests.

The theoretically optimal methods (in the sense of the computational complexity of the root finding problem) are of geometric type (\cite{renegar}, \cite{pan96}, \cite{neff-reif} \cite{schonhage}), but they are not widespread in practice. The implementation of GM  is usually enriched with some iterative stages, like Newton when the quadtree search reaches a subregion with fast convergence. Other changes in practical implementations of  a GM blueprint  are related with numerical issues of the inclusion tests. In the implementations in floating-point arithmetic, that is,
with a fixed finite precision, the roundoff errors can accumulate \cite{gourdon}. This can make that a test miss a root if, for example, the test indicate the presence of a root with a zero  value and its absence with a non zero value.  However, if we switch to an implementation in arbitrary-precision arithmetic, with no roundoff errors, other drawback arises: the precision (and hence the computational resources needed) can grow indefinitely \cite{pan}.
This problems of numerical accuracy have been an obstacle to the practical use of GM. 

In this work we expose a method with a floating point implementation that avoids the precision problem noted above. It is based on an inclusion test, IPSR in figure \ref{IPSR}, that is applicable to regions of arbitrary shapes, and on a recursive procedure giving a quadtree search, RDP in figure \ref{recur}. The inclusion test uses the winding number of closed plane curves. This number is the amount of twists of the curve around the origin. We have developed \cite{garciaart1} an efficient procedure to compute the winding number, and we use it in IPSR. Besides, we give a bound of the computational cost of the test. 
 In the recursive procedure RDP, we avoid the situations that require the use of IPSR with high precision, and hence also avoid the precision problem.
We compute a bound of the cost of the full root finding method using the cost bound of the inclusion test ISPN in each node of the RDP quadtree search.


About the structure of the paper, section 2 below exposes the procedure of winding number computation. We give results about its applicability and cost, mainly in terms of the distance to the origin from the curve to which the procedure applies. In section 3 we adapt these results to the curves that appear in the root finding method. The cost bound in this case is more complex, in terms of the distance to every root from the curve. The section 4 describe the quadtree procedure, RDP. The recursive partition ought to be performed avoiding to cross over polynomial roots, because this cause precision problem in IPSR. The section 5 contains a cost analysis of the full root finding method.

\section{Winding Number as inclusion test}

We compute the winding number of a curve using a discrete sampling obtained by an iterative procedure. We have proved the convergence of this procedure for nonsingular curves in \cite{garciaart1}, and we also modified it in such a way that the singular curves will not cause a endless computation.  

Let us consider closed curves defined parametrically, that is, as mappings of an interval to the complex plane  $\D:[a,b] \rightarrow \C$ with $\D(a)= \D(b)$. The {\em winding number}, or {\em index}, $\ind(\D )$ of a curve  $\D:[a,b] \rightarrow \C$, is the number of complete rotations, in counterclockwise sense, of the curve around the point $(0,0)$. See figure \ref{winnum}.
 
\begin{figure}[h!]
  \centering \includegraphics[width=\textwidth]{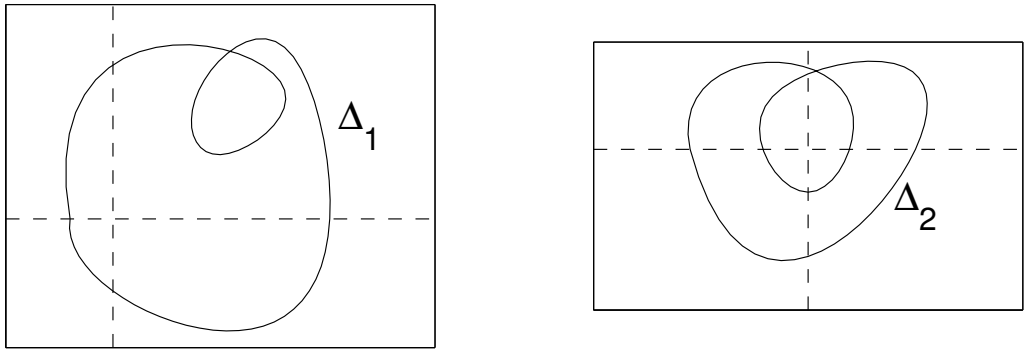}
  \parbox{.9\textwidth}{\caption{The winding numbers of the curves  $\D_1$ and  $\D_2$ are $\ind(\D_1) = 1$ and  $\ind(\D_2) = 2$}}.
  \label{winnum}
\end{figure}

\noindent As a particular case of the Cauchy formula of Complex Analysis, the winding number is equal to this line integral:

$$ \ind(\D)=\frac{1}{2\pi i}\int_\D{\frac{1}{w}dw}
$$
 
The principle of the argument \cite{henrici}  states that the number of zeros (counted with multiplicity) of an analytic function $f: \C \rightarrow \C$, $w=f(z)$, inside a region with border defined by the curve $\G,$ is equal to the winding number of the curve $\D=f(\G )$. The principle of the argument can be viewed as a bi-dimensional analogy of Bolzano's theorem, and it is in the base of several recursive methods to find the zeros of holomorphic functions \cite{henrici-search}.

It should be noted that the winding number of the curve $\D$  is not defined if $\D$  crosses over the origin $(0,0)$. In that case $\D$  is called {\em singular curve}, since the integral $\int_\D{\frac{1}{w}dw}$  does not exist. If $\D=f(\G )$, this is equivalent to that $\G$  crosses over a root of $f$.  In that case $\G$  is called singular curve with respect to $f$, the integral $\int_\D{\frac{1}{w}dw}$  does not exist, and the above theorem is not applicable.

We apply the principle of the argument to analytic functions defined by polynomials. The number of zeroes $N_0$ of a polynomial complex function $f(z)$ inside a region bounded by a curve $\G$  is the winding number of the curve $f(\G )$, that is:  $N_0=\ind(f(\G ))$. For an example, see figure \ref{winzcubo}. This can be used as an inclusion test to decide if a given region of the plane has any root, that is, $N_0 > 0$.

\begin{figure}[h!]
  \centering \includegraphics[width=\textwidth]{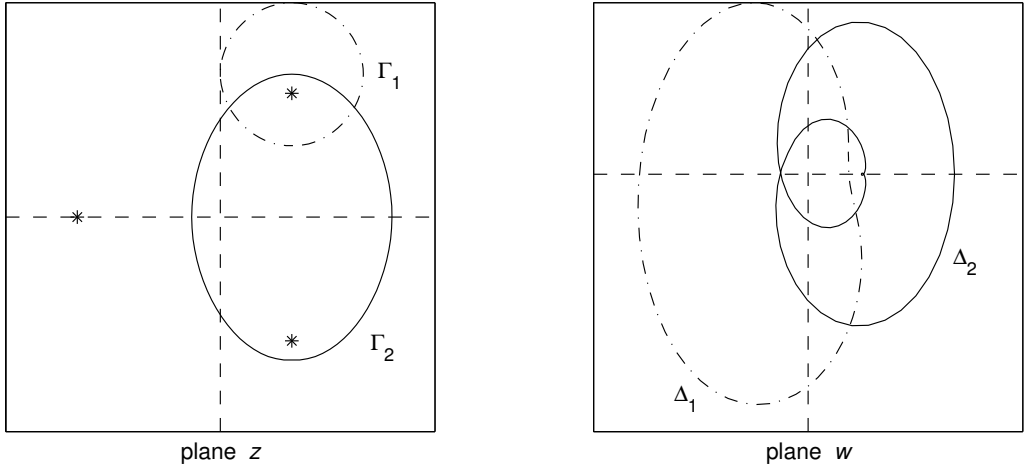}
  \parbox{.9\textwidth}{\caption{The number of roots of the polynomial $f(z)=z^3 + 1$ inside  $\G_1$ and  $\G_2$ equals the winding numbers of  $\D_1=f(\G_1)$ and $\D_2=f(\G_2)$.} \label{winzcubo}}
\end{figure}

The computation of the winding number is not an easy task. We could say that there are two schools of thought: the approaches that use the numerical integration,  as  \cite{kravanja}, and those based on computational geometry (\cite{henrici}). The first school use  the integration error estimates to ensure the correct calculation of the winding number. The second school, while producing faster algorithms, need to assume some hypothesis to prove their correctness. We follow this  school in the method of \cite{garciaart1}, and proof its correction under a mild hypothesis of Lipschitzianity. We now summarize the method and give a bound of its cost.

A curve $\D$ is defined by an continuous mapping  $\D:[a,b] \rightarrow \C$. We also call $\D$ to the set of its points, so $x\in\D$ means that there is $t\in[a,b]$ with $x=\D(t)$. Numerically, it is represented by  a polygonal approximation, that is, a discrete set of complex points, followed in certain order. Figure \ref{aprox} shows the curve $\D$   with two polygonal approximations.

\begin{figure}[h!]
  \centering \includegraphics[width=\textwidth]{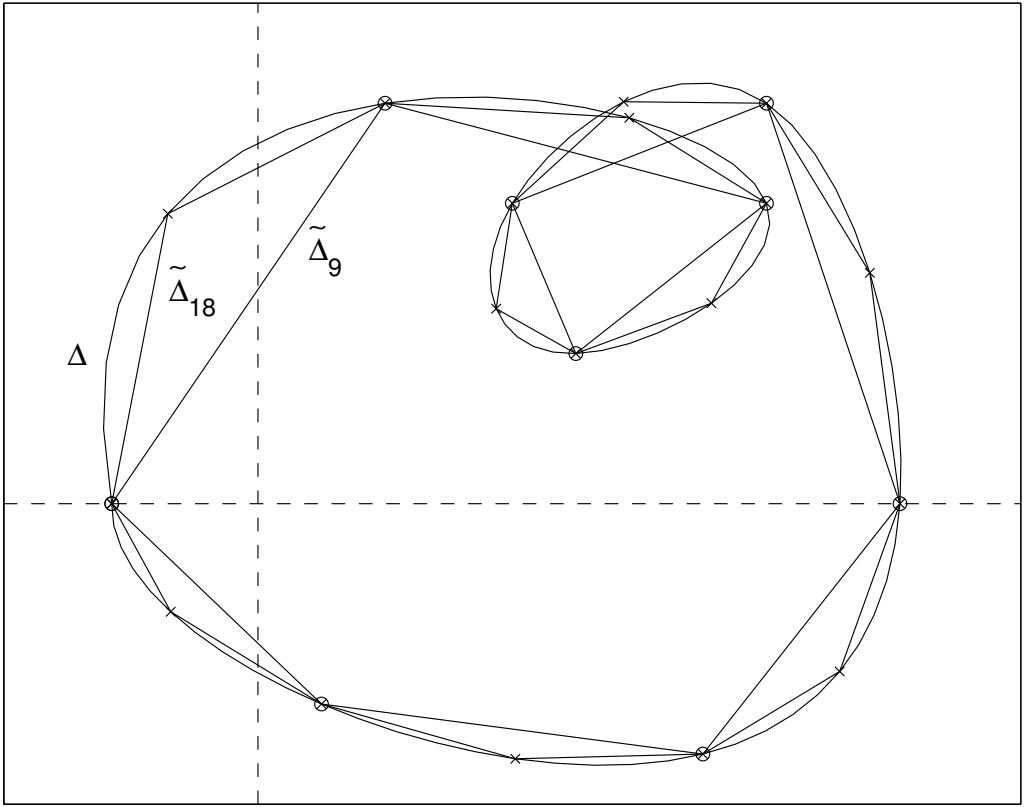}
  \parbox{.9\textwidth}{\caption{The curve  $\D$  is approximated by polygonal $\widetilde{\D_9}$ with a resolution of 9 points, and by $\widetilde{\D_{18}}$  with double resolution.}\label{aprox}}
\end{figure}

To calculate the winding number $\ind(\D )$ of a contour $\D$, we divide the complex plane in angular sectors, for example of angle ${\pi}/{4}$, like in figure \ref{adja}. There are eight such sectors, called $C_0, C_1, C_2 , \dots, C_7$, each sector the half of a quadrant.  A division in angular sectors with  angle different of  ${\pi}/{4}$ is also possible. The following method and the proof of its correction are similar in this case, provided that there are as least four sectors, and only  the constants involved must be changed.

Let us suppose that the curve $\D$  has a defined index, that is, it does not cross the origin.  Then it is sure that there is an array $(t_0, t_1, \dots, t_m)$ of values of the parameter  $t\in[a,b]$, $a = t_0 < t_1 < \dots < t_m = b$ whose images by the mapping $\D$  are in adjacent sectors. 

\begin{figure}[h!]
  \centering \includegraphics[width=\textwidth]{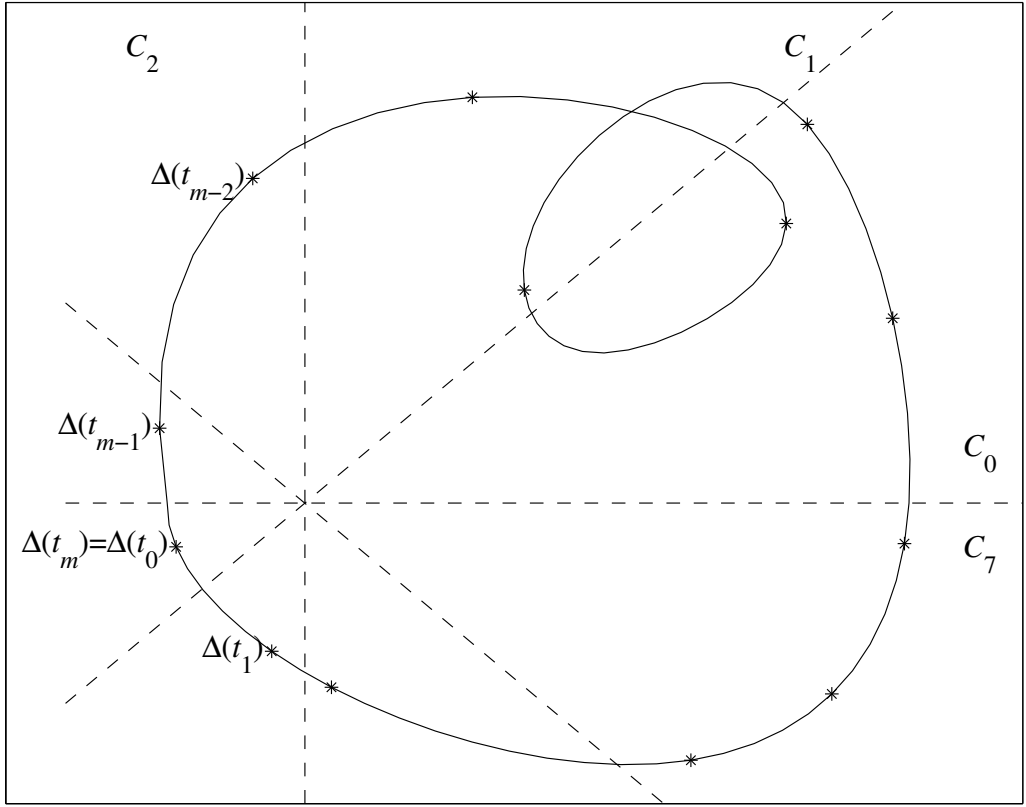}
  \parbox{.9\textwidth}{\caption{The images of the successive values $t_i$, $t_{i+1}$ are in adjacent (or the same) sectors.}\label{adja}}
\end{figure}

We say that a sector is {\em connected} to other sector if they are adjacent or are the same. Two points are {\em connected} if they are placed in connected sectors. A polygon of vertexes $\D(t_i)$, $i = 0, 1, \dots, m$ {\em satisfies the property of connection} if any two consecutive points $\D(t_i)$  and  $\D(t_{i+1})$ are connected. That is, the polygon verifies that if $\D(t_i)\in C_x$  and $\D(t_{i+1})\in C_y$  then $y = x \pm 1$ (or $y = x$). As the sectors $C_7$ and $C_0$ are adjacent, the equalities must be understood as an arithmetic congruence modulo 8.

Henrici \cite{henrici} suggests that if we obtain by some method an array of points $(\D(t_0),  \D(t_1), \dots,  \D(t_m))$ defining a polygon $\widetilde{\D_m}$ that verifies the property of connection, then its index $\ind(\widetilde{\D_m})$ can be calculated as the number of points $\D(t_i)$  in $C_7$ followed by a point $\D(t_{i+1})$  in $C_0$. The occurrence of a sector $C_0$ followed by $C_7$ must be counted negatively, standing for a turning back in the curve. That is, $\ind(\widetilde{\D_m})= \#(\mbox{crossings $C_7$ to $C_0$}) - \#(\mbox{crossings $C_0$ to $C_7$})$. The procedure to find the parameter values $(t_0, t_1,  \dots , t_m)$ was left unspecified by Henrici, as well as the conditions under which the index of the approximating polygon  $\widetilde{\D_m}$, defined by the $m$ points array $(\D(t_0),  \D(t_1), \dots,  \D(t_m))$, coincides with that of the original curve  $\D$. 
 
 The paper \cite{ying} of Ying and Katz devises a procedure to find an array $(\D(t_0),  \D(t_1), \dots,  \D(t_m))$ verifying the property of connection, at a reasonable computational cost. It constructs such array starting from some initial samples $(a = s_0, ..., s_n = b)$ of the curve  $\D$, whose images do not necessarily verify the property of connection. This means that perhaps, for some $i$, the images of $s_i$ and $s_{i+1}$ are not placed in connected sectors. The array of the images of $(s_0, \dots, s_n)$ is scanned from the beginning $s_0$, until a pair $(s_i, s_{i+1})$ is found such that its images $\D(s_i)\in C_x$  and $\D(s_{i+1})\in C_y$ are in non-connected sectors. In this situation an interpolation point $\frac{s_i + s_{i+1}}{2}$ is inserted in the array $(s_0, \dots, s_n)$ between $s_i$ and $s_{i+1}$. Then the resulting array  $(s'_0, ..., s'_{n+1})$ is scanned again for another pair $(s'_j, s'_{j+1})$  whose images are not in connected sectors and a middle point is inserted as described. Iterating this process, finally an array $(t_0, \dots, t_m)$, $m\geq n$, is obtained such that the polygon of its images $(\D(t_0), \dots,\D(t_m))$ verifies the property of connection.

This basic procedure of Ying-Katz has two features that prevent its practical application. First, it can run forever if it is applied to singular curves. Besides, in nonsingular curves, it can produce a result that is not correct, because a curve segment $\D(s_i, s_{i+1})$ that surrounds the origin can have its endpoints connected, hence is undetected by the procedure. For example, in figure \ref{lostturn}, $\D$ and $\widetilde{\D}$ have different winding number, by a lost turn between $s_i$ and $s_{i+1}$. 

\begin{figure}[h!]
  \centering \includegraphics[width=\textwidth]{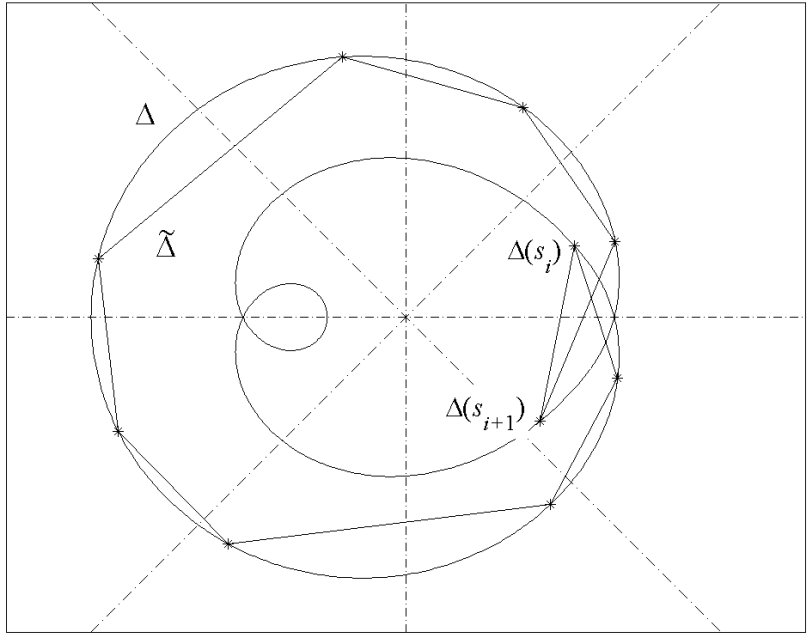}
  \parbox{.9\textwidth}{\caption{The curve  $\D$  is approximated by polygonal $\widetilde{\D}$, that verifies the property of connection, but theirs winding numbers are 2 and 1 respectively.}\label{lostturn}}
\end{figure}

In \cite{garciaart1} we proposed a procedure that enriches the loop of scanning-insertion of the basic procedure of Ying-Katz with an additional entry condition, that ensures the correct calculation of the winding number by the net number of $C7-C0$ crossings of the result array.
Being precise, for a curve  $\Delta:[a,b] \rightarrow \C$ Lipschitz of constant $L$, and an array $S =(s_0 , s_1 , \dots, s_n )$, the assertion $p(s_i)$ is  ``the values $s_i$  and $s_{i+1}$  in array $S$   have their images  $\Delta(s_i)$  and  $\Delta(s_{i+1})$  not connected'' and the assertion $q(s_i)$ is ``the values $s_i$  and $s_{i+1}$  in array $S$ verify   $(s_{i+1}-s_i)\geq \frac{|\Delta(s_i)|+ |\Delta(s_{i+1})|}{L}$''. The meaning of $s_{i+1}$ in the extreme case of $i=n$ is $s_0$. With this notation, we have the Insertion Procedure (so called IP) of figure \ref{IP}.

\begin{figure}[h]
  \centering 
  \fbox{
  \begin{minipage}{.9\textwidth}
  \vspace{.25cm}
  \hspace{0cm} \textbf{\textsl{Insertion procedure:}} To find the winding number of a curve $\Delta :[a,b] \rightarrow \C$

	\hspace{0.5cm} {\bf Input parameters:} The curve  $\Delta$ of Lipschitz constant $L$, and an array  $S=(a=s_0, s_1, \dots, s_n=b)$, sampling of $[a,b]$.
	
	\hspace{0.5cm} {\bf Output:} An array that is valid to compute $\ind(\Delta)$.

  \hspace{0.5cm} {\bf Method:}
   
 	\hspace{1cm} While there is a  $s_i$   in  $S$   with   $p(s_i)$  or with  $q(s_i)$  do:
	     
  \hspace{1.5cm} \{	Insert $\dfrac{s_i+s_{i+1}}{2}$   between $s_i$  and  $s_{i+1}$; 

  \hspace{1.5cm}  \}
  
  \hspace{1cm} Return the resulting array.
  \vspace{.25cm}
  \end{minipage}
  }
  \caption{Insertion procedure (IP).}
  \label{IP}
\end{figure}

Informally, we can say that the insertion procedure involves a while loop that is repeated until a connected array (i.e., verifying the condition ``no $p$'') and without lost turns (``no $q$'') is obtained.

The validity of the output array to compute $\ind(\D )$ is proved in \cite{garciaart1}. In this work the cost of IP is also analyzed, using the notion of $\varepsilon$-singularity. For a value  $\varepsilon\geq 0$, we say that a curve is  {\sl $\varepsilon$-singular} if its minimum distance to the origin is  $\varepsilon$. The 0-singular curves are the curves previously called singular for the index integral. 
In fact, the following theorem (\cite{garciaart1}) bounds the number of iterations of IP. We denote $|S|=\max_{0\leq i \leq n-1}(s_{i+1}-s_i)$.

\begin{Theorem} \label{teorIP} If   $\Delta:[a,b] \rightarrow \C$ is  $\varepsilon$-singular with   $\varepsilon\neq0$, with a Lipschitzian parameterization of constant $L$, then the IP for the curve $\Delta$ with initial array $S$ concludes in less than  $\frac{4L(b-a)}{\pi\varepsilon}\left\lfloor\frac{L|S|}{\varepsilon}\right\rfloor \left\lceil\lg_2\left(\frac{4L|S|}{\pi\varepsilon}\right)\right\rceil+\left(\frac{4L|S|}{\pi\varepsilon}+1\right)\left\lfloor\frac{L(b-a)}{\varepsilon}\right\rfloor$  iterations, and the returned array gives us the winding number.
 \end{Theorem}

Pushing aside the constants, this cost bound is of order $O\left(\frac{1}{\varepsilon^2}\log\left(\frac{1}{\varepsilon}\right)+\frac{1}{\varepsilon^2} \right)=O\left(\frac{1}{\varepsilon^2}\log\left(\frac{1}{\varepsilon}\right)\right)$. However, this result is of little interest for the practical application of IP, because the value $\varepsilon$ is not known a priori.

The above theorem says us that the number of iterations is bounded by a decreasing  function of $\varepsilon$.
Reciprocally, we can deduce that if the number of iterations surpasses certain threshold, the value $\varepsilon$ should be low, that is, the curve should be near the origin. Following this viewpoint, we define another procedure that extends IP with a check to detect a small value of $\varepsilon$  and exits in that case (that corresponds to a curve close to singular). It is called insertion procedure with control of singularity (IPS). IPS introduces the new assertion  $r(s_{i},Q)$ that depends on a input parameter $Q$, as shown in figure \ref{IPS}.  $r(s_{i},Q)$ is defined as ``the values $s_{i} $  and $s_{i+1} $  in array $S $  verify  $s_{i+1} -s_{i} \leq Q$''. The parameter $Q$ controls the maximum number of iterations that IPS can reach.

\begin{figure}[h]
  \centering 
  \fbox{
  \begin{minipage}{.9\textwidth}
  \vspace{.25cm}
  \hspace{0cm} \textbf{\textsl{Insertion procedure with control of singularity:}} To find the winding number of a curve $\Delta :[a,b] \rightarrow \C$

	\hspace{0.5cm} {\bf Input parameters:} The curve  $\Delta$ with Lipschitz constant $L$, an array  $S=(s_0, s_1, \dots, s_n)$, sampling of $[a,b]$, and a real parameter $Q>0$.
	
	\hspace{0.5cm} {\bf Output:} An array that is valid to compute $\ind(\Delta)$, on normal exit, or a value $t$ in $[a,b]$ such that $|\Delta(t)|\leq\frac{LQ}{\sin(\pi/8)}$, if exit on error.

  \hspace{0.5cm} {\bf Method:}
   
 	\hspace{1cm} While there is a  $s_i$   in  $S$   with   $p(s_i)$  or with  $q(s_i)$  do:
	     
  \hspace{1.5cm} \{	Insert $\frac{s_i+s_{i+1}}{2}$   between $s_i$  and  $s_{i+1}$; 
  
  \hspace{1.5cm}    If  $r(s_{i},Q)$, return either  $t=s_i$ or $t=s_{i+1}$ depending on $\min(|\Delta(s_i^{(k)})|,|\Delta(s_{i+1}^{(k)})|)$. [Exit on error]

  \hspace{1.5cm}  \}
  
  \hspace{1cm} Return the resulting array. [Normal exit]
  \vspace{.25cm}
  \end{minipage}
  }
  \caption{Insertion procedure with control of singularity (IPS).}
  \label{IPS}
\end{figure}

In \cite{garciaart1} is also proved that:

\begin{Theorem}\label{teorIPS} If  $\D:[a,b] \rightarrow \C$ is Lipschitzian with constant $L$, $S$ an array sampling of $[a,b]$, and $Q$ a positive real,  the insertion procedure with control of singularity applied to $\D$, $S$ and $Q$ verifies:
\begin{itemize}
\item[a)]	It returns in less than $\left\lfloor\frac{b-a}{Q}\right\rfloor$  iterations.
\item[b)]	If it exits normally then the returned array gives us $\ind(\Delta )$.
\item[c)]	If it exits on error, the returned value $t$ verifies  $|\Delta(t)|\leq\frac{LQ}{\sin\left(\frac{\pi}{8}\right)}$. 
\end{itemize}
\end{Theorem}

Theorem \ref{teorIPS} sets that IPS effectively computes the index of curves $\Delta$ with   $\varepsilon > \frac{LQ}{\sin\left(\frac{\pi}{8}\right)}$. For curves whose distance to the origin is under this level, IPS can by chance return normally with an array valid to compute the index, or it can return with error, signaling that the input curve is $\varepsilon$-singular. In any case, the procedure returns in less than $\left\lfloor\frac{b-a}{Q}\right\rfloor$ iterations.

\section{Insertion procedure for curves $\D= f(\G )$}

The computational cost of a winding number computation using IP (that is, the number of iterations) is related to the distance of the curve from the origin $\di(O,\D)$ (theorem \ref{teorIP}). The IPS, in addition, when the cost exceeds a threshold, returns with error with a bound on this distance (theorem \ref{teorIPS}). 

In this section we focus on curves of the form $\D=f(\G)$. These curves arise when we compute the number of roots of $f$ contained inside $\G$. 
In this particular case, the cost of applying IP to $\D$ given by theorem \ref{teorIP}  can be put in function of the distance of the curve $\G$ from the nearest root of $f$. We will show that the bound of this distance that we can obtain from theorem \ref{teorIPS} is not adequate for our purpose. In consequence, we define another procedure (IPSR, theorem \ref{teorIPSR}) that, when it returns on error, gives us a  bound of  certain function of the roots  and the curve (the condition number). This bound  will be useful in the next section, for the iterative method of partition of the search region. 

To compute the winding number of curves $\D=f(\G)$  using IP, the theorem \ref{teorIP} requires that $\D$ be Lipschitzian. To ensure that this requirement is fulfilled, we impose that $\G$ be uniformly parameterized. This is not a restriction in practice because the curve $\G$ that encloses the region of interest is usually built connecting straight segments (or circumference arcs) uniformly parameterized. Also, if this plane region is compact, its boundary $\G$ is bounded and, consequently, by a result about differentiable functions in a bounded set (\cite{kolmogorov}), $f$ is Lipschitzian (on the set  $\G$) with constant $L=\sup_{x\in\G}{|f'(\G(x))|}$. Hence  $\D= f(\G)$ is the composition of an uniformly parameterized curve and a Lipschitzian function, and so $\D$ is a Lipschitzian curve with the same constant $L$. Therefore the hypothesis of theorem \ref{teorIP} is verified, and IP computes the winding number of $\D= f(\G )$ (i.e., the number of roots inside $\G$).

The key factor in the cost of the winding number computation, by theorem \ref{teorIP}, is the distance to the origin from $\D$. This distance and the roots of $f$ are related as described in the following proposition.

Let $f(z)=a_nz^n+\dots+a_1z+a_0$ be a polynomial of degree $n$ and its root decomposition $f(z)=a_n(z-z_1)(z-z_2)\dots(z-z_n)$. Recall that the distance $\di(A,B)$ between two sets $A$ and $B$ is the minimum of the distances between each pair of points of the respective sets. In particular, if $Z$ is the set of the roots of $f$, $Z=\{z_1,z_2,\dots, z_n\}$, its distance $\di(Z,\G)$ to  $\G$ is the distance from the closer root, and $\di(O,\D)$ is the distance from the origin $O$ to $\D$.

\begin{propositionn}\label{GtoD}\label{DtoG} If  $\G:[a,b] \rightarrow \C$ is a curve uniformly parameterized, $f$ a polynomial of degree $n$,  with Lipschitz constant $L$, and $\D=f(\G)$, then:
$$
|a_n|\di(Z,\G)^n\leq \di(O,\D)\leq L\di(Z,\G)
$$
\end{propositionn}

Remember that if we apply IP to compute the number of roots inside $\G$, its cost will be $O\left(\frac{1}{\varepsilon^2}\log\left(\frac{1}{\varepsilon}\right)\right)$ by the bound of Theorem \ref{teorIP}, where $\varepsilon=\di(O,\D)$. In terms involving $\G$, by the inequality   
$\frac{1}{|a_n|\di(Z,\G)^n}\geq\frac{1}{\di(O,\D)}$ deduced from the above proposition, the computation has a cost of order lesser or equal than  $$O\left(\frac{1}{(|a_n|\di(Z,\G)^n)^2}\log\left(\frac{1}{|a_n|\di(Z,\G)^n}\right)\right)$$ that is $O\left(\frac{1}{\di(Z,\G)^{2n}}\log\left(\frac{1}{\di(Z,\G)}\right)\right)$.

As the distance $\di(Z,\G)$ to the roots is unknown, we confront an undetermined cost using IP. This was the motivation to introduce IPS (figure \ref{IPS}), which with we have an assured bound of cost of $\left\lfloor\frac{(b-a)}{Q}\right\rfloor$ iterations, by theorem \ref{teorIPS}, since $\D=f(\G)$ and $\G$ is defined over $[a,b]$. In addition to this assured cost, a second advantage of IPS is that if the bound of cost is reached, it returns on error with a point of the curve near to origin, $|\Delta(t)|\leq\frac{LQ}{\sin\left(\frac{\pi}{8}\right)}$, as theorem \ref{teorIPS} says. This gives us a bound to $\di(O,\D)$, that is $\frac{LQ}{\sin\left(\frac{\pi}{8}\right)}$. In the particular case of curves of the form $\D=f(\G)$, we can deduce a bound to $\di(Z,\G)$:  by the proposition \ref{GtoD},   $|a_n|\di(Z,\G)^n\leq \di(O,\D) \leq |\Delta(t)| \leq \frac{LQ}{\sin\left(\frac{\pi}{8}\right)}$, and hence $\di(Z,\G)\leq\sqrt[n]{\frac{LQ}{|a_n|\sin(\pi/8)}}$, in case of error exit of IPS.


%

In case of return with error we need a  bound finer than this on $\di(Z,\G)$, to locate a root close to the point that produce the error. Unfortunately, the bound given by IPS is very loose, because it  decreases much more slowly than the parameter $Q$: its value can be near to 1 for moderate values of $n$ even with a very small $Q$. Taking the $n$-th root is a severe handicap in the formula of a bound. The procedure IPS, as least to the extent that is covered by theorem \ref{teorIPS}, does not provide this bound at a reasonable cost. 

We must define another procedure for which a bound exists that depends on $Q$ without $n$-th root extraction. In the Insertion Procedure with control of Singularity for number of Roots (IPSR, figure \ref{IPSR}), we use the predicates $p$, $q_2$ and $r$ meaning $p(s_i)$  that ``the values $s_i$  and $s_{i+1}$  in array $S$   have their images  $f(\G(s_i))$  and  $f(\G(s_{i+1}))$  not connected'', $q_2(s_i)$ is ``the values $s_i$  and $s_{i+1}$  in array $S$ verify: 
\begin{align*}
\left|f(\G(s_i))\right| +\left| f(\G(s_{i+1})) \right|\leq2\left|f'(\Gamma(s_i)) \right|(s_{i+1}-s_i)& + 
\left|f(\G(s_{i+1}))- f(\G(s_{i}))\right|
\text{''},
\end{align*}
and $r(s_i,Q)$ means ``the values $s_{i}$  and $s_{i+1} $  in array $S$  verify
$(s_{i+1} -s_{i})\leq Q$''. As before, $s_{i+1}$ in the extreme case of $i=n$ should be intended as $s_0$. The assertions $p$ and $r$ are equal than in IPS, but $q_2$ is different from $q$. The assertion $q_2$ is required in proposition \ref{lt}   to prove that there are not lost turns and in proposition  \ref{qins} to show that certain bound in error return is verified.   


The return of the IPSR verifies a claim that involves $\kappa_f(\G)$, the {\em condition number of the curve $\G$  with respect to $f$},  defined as the sum of inverses of the distances from $\G$ to each root of $f$:
$$\kappa_f(\G)=\sum_{i=1}^{n}\frac{1}{\di(z_i,\G)}.$$ 

\begin{figure}
  \centering 
  \fbox{
  \begin{minipage}{.9\textwidth}
  \vspace{.25cm}
  \hspace{0cm} \textbf{\textsl{Insertion procedure with control of singularity for the number of roots:}} To find the number of roots of a polynomial $f$ inside a curve $\G :[a,b] \rightarrow \C$

	\hspace{0.5cm} {\bf Input parameters:} The curve  $\G$ uniformly parameterized, the polynomial $f$ of degree $n$ and Lipschitz constant $L$, an array  $S=(s_0, \dots, s_n)$, sampling of $[a,b]$, and a real parameter $Q>0$.
	
	\hspace{0.5cm} {\bf Output:} An array $(t_0, \dots, t_m)$ that is valid to compute $\ind(f(\G))$ on normal exit. A value $t \in [a,b]$ such that   $\kappa_f(\G)$ is greater  than  $\frac{\sin(\pi/8)}{Q}$, if exit on error. 

  \hspace{0.5cm} {\bf Method:}
   
 	\hspace{1cm} While there is a  $s_i$   in  $S$   with   $p(s_i)$  or with  $q_2(s_i)$  do:
	     
  \hspace{1.5cm} \{	Insert $\frac{s_i+s_{i+1}}{2}$   between $s_i$  and  $s_{i+1}$; 
  
  \hspace{1.85cm}    If  $r(s_{i},Q)$, return. [Exit on error]

  \hspace{1.5cm}  \}
  
  \hspace{1cm} Return the resulting array. [Normal exit]
  \vspace{.25cm}
  \end{minipage}
  }
  \caption{Insertion procedure with control of singularity for number of roots (IPSR).}
  \label{IPSR}
\end{figure}

The role of condition number  $\kappa_f(\G)$ in IPSR is similar to that of the inverse of $\di(O,\D)$, the singularity of $\D$, in IPS. By theorem \ref{teorIPS}, an error return of IPS implies a low value of $\di(O,\D)$, that is, a high value of $\frac{1}{\di(O,\D)}$. We will show that an error return of IPSR implies a high value of $\kappa_f(\G)=\sum_{i=1}^{n}\frac{1}{\di(z_i,\G)}$.

With the following proposition we will prove (in theorem \ref{teorIPSR}) that if IPSR exit normally, the returned array gives the number of roots inside $\G$.

\begin{propositionn} \label{lt} For a curve $f(\G)$, if $q_2(s_i)$ in array $S=(\dots,s_i,s_{i+1},\dots)$ is not verified, then there is not a lost turn between $s_i$ and $s_{i+1}$.
\end{propositionn}

We now classify the iterations performed by the ``while'' loop. An iteration is {\sl of type $p$} if it is performed because the property $p(s_i)$  is verified, and  it is {\sl of type $q$} if it is performed because the property ``not $p(s_i)$, and $q_N(s_i)$'' is verified. So any iteration is of type $p$ or of type $q$, but not both. Likewise the error exits can be classified as of type $p$, or $q$,  according to the type of the iteration in which the return occurs. 
 
We will show in theorem \ref{teorIPSR} that IPSR exits with an array valid to compute the number of roots inside $\G$, or with a lower bound of the condition number (that is, the analogous of theorem \ref{teorIPS} for IPS). Previously we give bounds on this condition  on errors of type $p$ and $q$ in two propositions after several auxiliary lemmas. We consider the angles measured in the interval $[0,2\pi)$.

\begin{lemman} \label{xyz} If $\alpha$ is the oriented angle between three complex numbers $(x,z,y)$ with vertex at $z$ (that is, $\alpha=\arg(z-y)-\arg(z-x)$) then 
$$
\min(\di(z,x),\di(z,y))\leq\frac{\di(x,y)}{2|\sin({\alpha}/{2})|}
$$
\end{lemman}

The accuracy and cost of the IPSR is described by the following theorem:

\begin{Theorem} \label{teorIPSR} If  $\G:[a,b] \rightarrow \C$  
is uniformly parameterized, $S$  an array sampling of $[a,b]$, $Q$ a positive real, and $f$ a polynomial of degree $n$ with Lipschitz constant $L$ in $\G$, the insertion procedure with control of singularity, IPSR, verifies:
\begin{itemize}
\item[a)]	Returns in less than $\left\lfloor\frac{b-a}{Q}+1\right\rfloor$  iterations.
\item[b)]	If it exits normally then the returning array gives the number of roots of $f$ inside $\G$.
\item[c)]	If it exits on error, then $\kappa_f(\G)\geq\frac{\sqrt{2}}{4Q}$.
\end{itemize}
\end{Theorem}

The condition number $\kappa_f(\G)$ and the distance of $\G$ from the nearest root of $f$ are related in the following way: 

\begin{propositionn} \label{distkappa} If $\kappa_f(\G)\geq \frac{\sqrt{2}}{4Q}$ then there is a root $z_i$ with $\di(z_i,\G) \leq \frac{4nQ}{\sqrt{2}}$.
\end{propositionn}

The insertion procedure with control of root proximity prevents an excessive number of iterations, controlled by input parameter $Q$. The IPSR effectively computes the number of roots inside $\G$ if there are not close roots. But if there are roots on $\G$ or near it, the procedure can return with error, signaling this fact, or can return normally with an array valid to compute the number of roots, always in less than $\left\lfloor\frac{b-a}{Q}+1\right\rfloor$  iterations. 

Perhaps it is also possible to proof that it returns normally in less than $O(\kappa_f(\G)^2\log(\kappa_f(\G)))$ iterations, a result similar to theorem \ref{teorIP}. It should be noted that this dependence on the distance to the  roots is consistent with other algorithms that compute the number of roots in a region  \cite{renegar,pan}. In any case, this use  of the condition number $\kappa_f(\G)$ for  theoretical analysis  is not necessary for our purposes of root finding.
%

\section{Recursive use of the procedure}
\label{recursec}

In this section we detail a geometric algorithm for root finding, following the common pattern depicted in the introduction, that is, an inclusion test to decide if there is a root in a region, and a recursive procedure (RDP, figure \ref{recur}), to subdivide the region and locate the roots with the  needed accuracy. 
With the result of theorem \ref{teorIPSR}, the IPSR computes the number of roots inside $\G$ by the winding number of $\D=f(\G)$. This procedure will be used as inclusion test. We now describe the recursive procedure.

The size of a region $P$ is measured by its rectangular diameter $\dr(P)$, that is precisely defined in subsection \ref{splitting}. An {\em approximation of a root up to an accuracy of $A>0$} is a region of rectangular diameter lesser than $A$ that contains the root.

\begin{figure}
  \centering 
  \fbox{
  \begin{minipage}{.9\textwidth}
  \vspace{.25cm}
  \hspace{0cm}\textbf{\textsl{Recursive Division Procedure:}} To find the roots of a polynomial $f$ inside a region $P$.

	\hspace{0.5cm} {\bf Input parameters:} A region $P$ of the complex plane, a polynomial $f$, and a parameter of accuracy $A>0$.
   	
	\hspace{0.5cm} {\bf Output:} The array $R=(P_1,P_2,\dots,P_k)$ (containing $k$ approximations   up to an accuracy of $A>0$  of all the roots of $f$  in $P$) and the array $N=(n_1,n_2,\dots,n_k)$ (of numbers  $n_i\geq 1$ such that each $P_i$ contains $n_i$ roots of $f$ counting multiplicity).
 
   \hspace{0.5cm} {\bf Global structures used:} Two arrays $R$ and $N$, initially empty.

   \hspace{0.5cm} {\bf Method:} $\RDP(P,A)$  \{

   \hspace{1cm} If windN$(P)=0$, return. [exit 1] 
   	     
   \hspace{1cm} If $\dr(P)<A$ then

   \hspace{1.5cm} Add $P$ to $R$ and windN$(P)$ to $N,$ and  return. [exit 2] 
   
	\hspace{1cm} Else

   \hspace{1.5cm} Divide $P$ in four subregions $P_1,P_2,P_3$ and $P_4$. 

   
   

   \hspace{1.5cm} For each $i=1,2,3,4$, do 

   \hspace{2cm} $\RDP(P_i,A)$.     

   \hspace{1cm}  [exit 3]  

  \hspace{0.5cm}  \}

  \vspace{.25cm}
  \end{minipage}
  }
  \caption{Recursive Division Procedure (RDP)  for polynomial root finding. Two items remain to be specified: the parameter $Q$ to compute windN$(P)$ using IPSR$(P,Q)$, and the decomposition method of ``Divide P''.}
  \label{recur}
\end{figure}

The Recursive Division Procedure (RDP, figure~\ref{recur}) divides the initial region into subregions progressively smaller, until the diameter goes below the accuracy. These regions are inserted into the array $R$, which at the end of RDP contains approximations to the roots. Each  region $P_i$ included in $R$ is labeled with the number of roots that it contains, $n_i$. The case of $n_i> 1$ corresponds to either a multiple root or a cluster (a set of roots with distance lesser than $A$ between them).


Two items of the RDP in figure~\ref{recur} must be precised. The first item is the computation of windN($P$), the winding number of the image by $f$ of the border of $P$. This computation is made by IPSR($P, Q$) for some value of $Q$, as previously mentioned. This value $Q$ must be chosen as a compromise between a large value (that gives us lower computational cost by theorem 3a), and a value small enough to not cause an error (Theorem 3b) in the initial region or in the subregions that arise in the recursive decomposition. The computation must be done without error because the winding number is needed for processing a region (that is, either discarding or decomposing it). 

The second item is the decomposition ``Divide $P$'', which makes four subregions $P_1, P_2, P_3 $ and $ P_4$ from $P$. This should be done with the following three requirements: a) the subregions obtained $P_1, P_2, P_3 $ and $ P_4$ must be a set partition of $P$ (that is, $P =\bigcup  P_i$ and $P_i \bigcap P_j = \emptyset$), so at the end each root will be in a component of $R$ and only in one, b) the diameter must be decreasing, 
$\dr(P_i)<\dr(P)$, and c) each region obtained $P_i$  must be such that IPSR($P_i, Q$) returns without error for the chosen value $Q$, as discussed above. We will  describe a decomposition method and a value of  $Q$ such that these requirements are fulfilled.

In the next subsection \ref{splitting} we make a first try of a decomposition method ``Divide P''. We cut successively in half the region, first horizontally and then vertically, and we will show that the resulting subregions are decreasing in rectangular diameter. However the arising subregions can cause an error in IPSR. In the following subsection \ref{managingerror}, we make a second and definitive try of method ``Divide P'', and show that a value of $Q$ lesser than $\frac{A\sin(\pi/8)}{n_0n4\sqrt{2}}$ does not produce error with IPSR in the subregions, being $n_0$ the number of roots in the original region. In the final subsection \ref{termination} we show that the RPD ends in a finite number of calls, that the claim on the output of the figure \ref{recur} is verified at the end, and we calculate its computational cost.


Before this, a comment about the type of regions that we consider. 
For simplicity, we require that the initial region, as well as those that arise in its subdivisions, are plane regions whose border is a simple closed curve $\G$ (Jordan curve \cite{kolmogorov}). These regions are connected. A non connected region is, for instance, the interior of two disjoint circles, and its border (two circumferences) is not a a simple curve. Non connected regions can be treated in a way  similar to the connected ones, because the number of roots inside a non connected region is the sum of the winding numbers of its connected components. However, the assumption of connectedness simplify the reasonings.

We will produce the subregions by cuts along straight lines. To prevent such cuts to cause a subregion non connected, we enforce that the initial region $P$ is 
convex (i. e., that if a straight segment has its endpoints in $P$, all the segment is inside $P$). So, the subregions will be convex (hence connected). To find the roots in a non convex region, the RDP should be applied to its convex hull, or to its decomposition into convex parts.


\subsection{Divide into Smaller Pieces: a first try}
\label{splitting}

We consider each plane region as a closed set, that is, including its border. For a plane region $P$, its {\em top horizontal supporting line} $l_T$ is the upper horizontal straight line that has some point in common with $P$. Likewise,  the {\em bottom horizontal supporting line} $l_B$ of $P$ is the lower horizontal straight line that has some point in common with $P$. The {\em vertical diameter} $\dm_V(P)$ is the distance between these lines $\di(l_{T},l_B)$. That is $\dm_V(P)=\min_{x\in l_T,\,y\in l_B}\di(x,y)$.

 In a similar manner, the {\em left vertical supporting line} $l_L$ is the leftmost vertical straight line that has some point in common with $P$, and the {\em right vertical supporting line} $l_R$ is the rightmost straight line that has some point in common with  $P$. The {\em horizontal diameter} $\dm_H(P)$ is the distance between these lines $\di(l_L,l_R)$. That is $\dm_H(P)=\min_{x\in l_L,\,y\in l_R}\di(x,y)$. The figure \ref{diameters} depicts these concepts, and also the classical {\em diameter} of $P$ defined by $\dm(P)= \max_{x,y\in P}\di(x,y)$. 

\begin{figure}[h!]
  \centering \includegraphics[width=.6\textwidth]{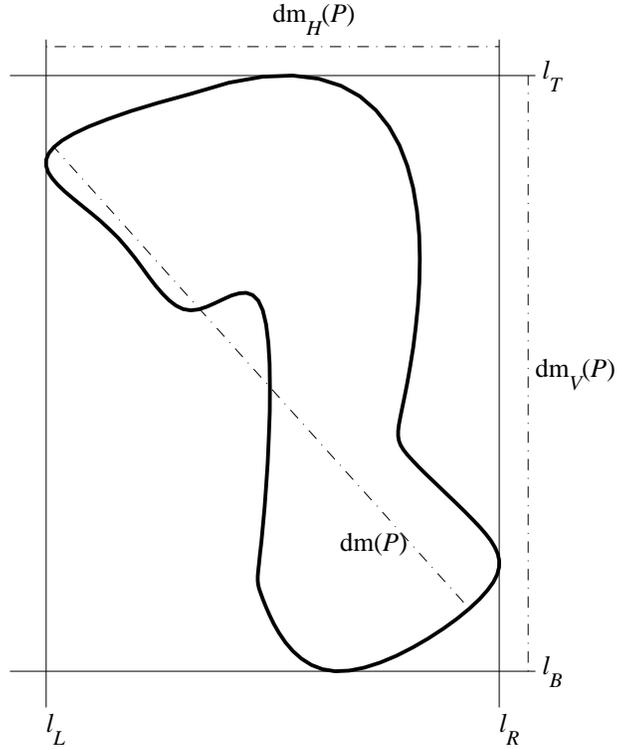}
  \parbox{.9\textwidth}{\caption{Supporting lines and diameters for a plane region $P$.}\label{diameters}}
\end{figure}

For regions defined by straight segments, as the polygonals arising in the IPSR, the horizontal and vertical diameters are easier to compute than the classical diameter. They are related by the following bounds.

\begin{lemman} \label{contain} If two plane regions $P_1,P_2$ verify $P_1 \subset P_2$, then $\dm(P_1)\leq\dm(P_2)$, $\dm_H(P_1)\leq\dm_H(P_2)$, and $\dm_V(P_1)\leq\dm_V(P_2)$.
\end{lemman}
\begin{proof}
For the classical diameter is clear because the maximum, in the definition of $\dm(P_2)$, is taken over a greater set than in $\dm(P_1)$. For the horizontal diameter, note that the two vertical supporting lines of $P_1$ are between the two vertical supporting lines of $P_2$, perhaps  coinciding with some of them. Hence its distance is lesser. Likewise for the vertical diameter.

\qedhere
\end{proof}

We also define the {\em rectangular diameter} $\dr(P)=\sqrt{\dm_H(P)^2 +\dm_V(P)^2}$. 

\begin{propositionn} \label{chain} For a plane region $P$, it is verified:
$$
\max(\dm_H(P),\dm_V(P))\leq\dm(P)\leq \dr(P)
$$
\end{propositionn}
\begin{proof}
For the first inequality, note that the supporting lines of $P$ have as least one point in common with $P$. Being $x_0$ one of these points for the upper horizontal supporting line, that is $x_0\in P\cap l_T$, and $y_0$ for the lower one, $y_0\in P\cap l_B$, using the definitions of $\dm_V(P)$ and $\dm(P)$ we have $\dm_V(P)=\min_{x\in l_T,\,y\in l_B}\di(x,y)\leq \di(x_0,y_0)\leq  \max_{x,\,y\in P}\di(x,y)=\dm(P)$. Likewise $\dm_H(P)\leq \dm(P)$, hence $\max(\dm_H(P),\dm_V(P))\leq\dm(P)$. 

For the second inequality, we define the {\em HV-envelope} of $P$, $\text{Env}_{HV}(P)$, as the bounded rectangle delimited by the horizontal and vertical supporting lines. As $P\subset\text{Env}_{HV}(P)$, by the lemma \ref{contain} we have that $\dm(P)\leq \dm(\text{Env}_{HV}(P))$. Furthermore, the diameter of the rectangle $\text{Env}_{HV}(P)$, of base $\dm_H(P)$ and height $\dm_V(P)$, is the distance between opposed vertexes, then $\dm(\text{Env}_{HV}(P))=\sqrt{\dm_H(P)^2 +\dm_V(P)^2}=\dm_{Rect}(P)$. Chaining with the previous inequality, we conclude.
\qedhere
\end{proof}

To divide a figure in lesser parts, we define the operators $T$, $B$, $L$ and $R$ acting on a plane region $P$. If $m_H(P)$ is the straight line in the middle between the horizontal supporting lines, $T(P)$ is the intersection of $P$ with the upper half-plane defined by $m_H(P)$, and $B(P)$ is the intersection of $P$ with the lower half-plane. Likewise, If $m_V(P)$ is the straight line in the middle between the vertical supporting lines, $L(P)$ is the intersection of $P$ with the left half-plane defined by $m_V(P)$, and $R(P)$ is the intersection of $P$ with the right half-plane. The operators $T$ and $B$ are said of type {\em horizontal}, while $L$ and $R$ are said of type {\em vertical}. The picture \ref{operators} shows several compositions of these operators applied to a non convex region.

\begin{figure}[h!]
  \centering \includegraphics[width=.4\textwidth]{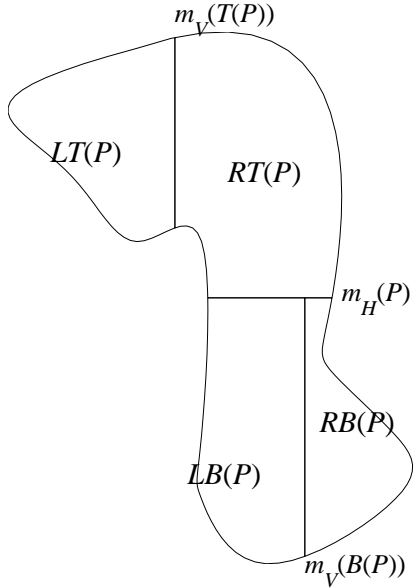}
  \parbox{.9\textwidth}{\caption{Applying the operators $T$, $B$, $L$, $R$ to  $P$.}\label{operators}}
\end{figure}

\begin{lemman} \label{half} The operators decrease the diameters, verifying:
$$\begin{array}{cc}
\dm_H(T(P))\leq \dm_H(P), & \dm_H(L(P))\leq \dfrac{\dm_H(P)}{2}, \\ 
 \dm_V(T(P))\leq \dfrac{\dm_V(P)}{2} & \text{and }\dm_V(L(P))\leq \dm_V(P) 
\end{array} 
$$
The same formulas are valid changing $T$ to $B$ and $L$ to $R$. 
\end{lemman}
\begin{proof}
For the horizontal diameter, the first inequality is a consequence  of lemma \ref{contain}. The second comes from that the  vertical supporting lines of $L(P)$ are between $l_L$ and $m_V(P)$, hence the horizontal diameter are lesser or equal than the distance $\di(l_L,m_V(P))$, that is the half of $\dm_H(P)$. 

For the vertical diameter, the reasoning is similar, and also for operators $B$ and $R$.

\qedhere
\end{proof}

We will find the rate of diameter decreasing under the  alternating applications of horizontal and vertical divisions to the initial region. An operator is {\em of type ${\cal O}_{m}$}, for $m\geq 1$, if it is the composition of $m$ operators of type horizontal alternating with $m$ of type vertical, starting with an horizontal operator. That is, $O$ is of type ${\cal O}_{m}$ if $O=V_mH_mV_{m-1}H_{m-1}\cdots V_1H_1$, being $H_i\in\{T,B\}$ and $V_i\in\{L,R\}$. There are $2^{2m}$ operators of type ${\cal O}_{m}$. (See figure \ref{ops23}).


\begin{figure}[h!]
  \centering \includegraphics[width=.5\textwidth]{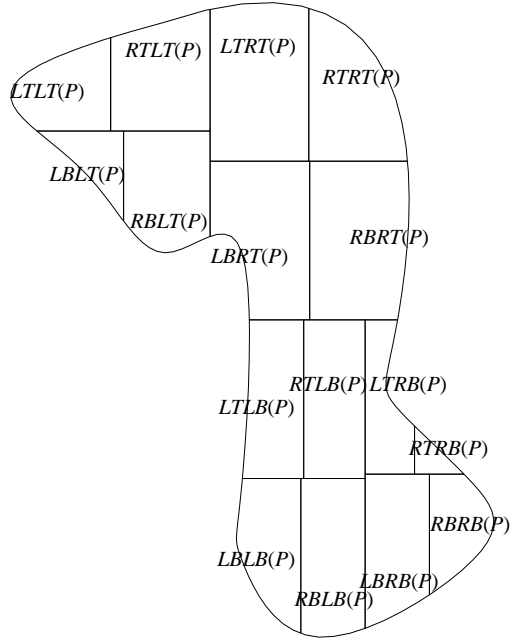}
  \parbox{.9\textwidth}{\caption{The regions arising from applying all the operators in ${\cal O}_{2}$ to a region.}\label{ops23}}
\end{figure}

\begin{lemman} \label{decrease} For $m\geq 1$, if $O_m$ is an operator of type ${\cal O}_{m}$
$$
\dr(O_m(P))\leq \dfrac{\dr(P)}{2^m}
$$
\end{lemman}
\begin{proof}
By induction on $m$. We denote as above with $H$ an operator that can be $T$ or $B$, and with $V$ another that can be $L$ or $R$. For $m=1$, we have that $O=VH$. Chaining some inequalities of lemma \ref{half} several times:
\begin{multline*}
\dr(VH(P))= \sqrt{\dm_H(VH(P))^2 +\dm_V(VH(P))^2} \leq \\ 
\leq \sqrt{\dfrac{\dm_H(H(P))^2}{2^2} +\dm_V(H(P))^2}  \leq 
\dfrac{\sqrt{\dm_H(P)^2 +\dm_V(P)^2}}{2}=\dfrac{\dr(P)}{2}
\end{multline*}

For $m>1$, note that if $O_m$ is of type ${\cal O}_{m}$, then  $O=VHO_{m-1}$ with  $O_{m-1}$ of type ${\cal O}_{m-1}$. Applying again proposition \ref{chain} and lemma \ref{half} several times,  we have:
\begin{multline*}
\dr(O_m(P))=\\
=\dr(VHO_{m-1}(P)) \leq \sqrt{\dm_H(VHO_{m-1}(P))^2 +\dm_V(VHO_{m-1}(P))^2} \leq \\
\leq \sqrt{\dfrac{\dm_H(HO_{m-1}(P))^2}{2^2}+\dm_V(HO_{m-1}(P))^2 } \leq \\
\leq \sqrt{\dfrac{\dm_H(O_{m-1}(P))^2}{2^2} +\dfrac{\dm_V(O_{m-1}(P))^2}{2^2}} 
\leq  \\ \leq
\dfrac{\sqrt{\dm_H(O_{m-1}(P))^2 +\dm_V(O_{m-1}(P))^2}}{2}
\end{multline*}
This, applying  the hypotesis of induction, is lesser or equal than: 
\begin{multline*}
\dfrac{
  \sqrt{\left(\dfrac{\dm_H(P)}{2^{m-1}}\right)^2 
       +\left(\dfrac{\dm_V(P)^2}{2^{m-1}}\right)^2
  }
}{2} 
= \\ =\frac{\sqrt{\dm_H(P)^2 +\dm_V(P)^2}}{2^m}=\frac{\dr(P)}{2^m} $$
\end{multline*}

\qedhere
\end{proof}

In particular, $\dr(O_{m+1}(P))\leq\frac{\dr(O_{m}(P))}{2} $. With respect to the  classical diameter, in general it is not true that $\dm(O_m(P))\leq \dfrac{\dm(P)}{2^m}$. For example if $P$ is the circle of radius 1, $\dm(P)=2$ but $TR(P)$ is a circular sector with $\dm(TR(P))=\sqrt{2}>\dfrac{\dm(P)}{2}$. However we have the following:

\begin{Corollaryn} If $O_m$ is an operator of type ${\cal O}_{m}$
$$
\dm(O_m(P))\leq \dfrac{\dr(P)}{2^m}
$$
\end{Corollaryn}
\begin{proof}
Applying  proposition \ref{chain}, $\dm(O_m(P))\leq \dr(O_m(P))$, and chaining with the above lemma. 

\qedhere
\end{proof}

By the above corollary, we have that the division of the region by horizontal and vertical cuts effectively reduces the diameter of the obtained parts. 

In the context of the recursive procedure RDP of figure \ref{recur}, we  define the decomposition in subregions ``Divide $P$'' as $P_1=RT(P)$, $P_2=LT(P)$, $P_3=RB(P)$ and $P_4=LB(P)$. That is, we first cut horizontally, and then $T(P)$ and $B(P)$ are cut vertically. By lemma \ref{decrease} for $m=1$, the obtained regions will satisfy the requirement that $\dr(P_i)<\dr(P)$. Besides, the lemma \ref{decrease} for $m>1$ also describes the size of regions of later recursive calls.

\subsection{Divide into Smaller Pieces: second try}
\label{managingerror}

The decomposition method described above produces subregions decreasing in rectangular diameter. However, if we apply IPSR to these subregions, it can
return on error without finishing the computation, as theorem \ref{teorIPSR} established. The following is an example of this situation. 

Let us consider a circumference $\G$ centered slightly at the right of the origin, that contains the roots of the polynomial $f(z)=z^3-1$. Its image $\D=f(\G)$ encircles three times the origin (see figure \ref{subreg}). The origin is marked with $\circ$, and it is the image of the three roots. We can suppose for the sake of the example that the IPSR, with a reasonable parameter $Q$, computes the winding number of $\D$, that is 3. 
Following RDP in figure \ref{recur}, we decompose the interior of $\G$ in four circular sectors, by an horizontal and a vertical cut. One of these subregions has a border crossing over a root, like the depicted in figure \ref{subreg} {\em c}). The image of the border of this subregion, say $f(\G_s)$, crosses over the origin (figure \ref{subreg} {\em d}), and hence the IPSR can not compute its winding number.

\begin{figure}[h!]
  \centering
  \begin{tabular}{cc}
\includegraphics[width=.50\textwidth,trim = 0mm .54\textwidth mm .65\textwidth mm 0mm,clip=true]{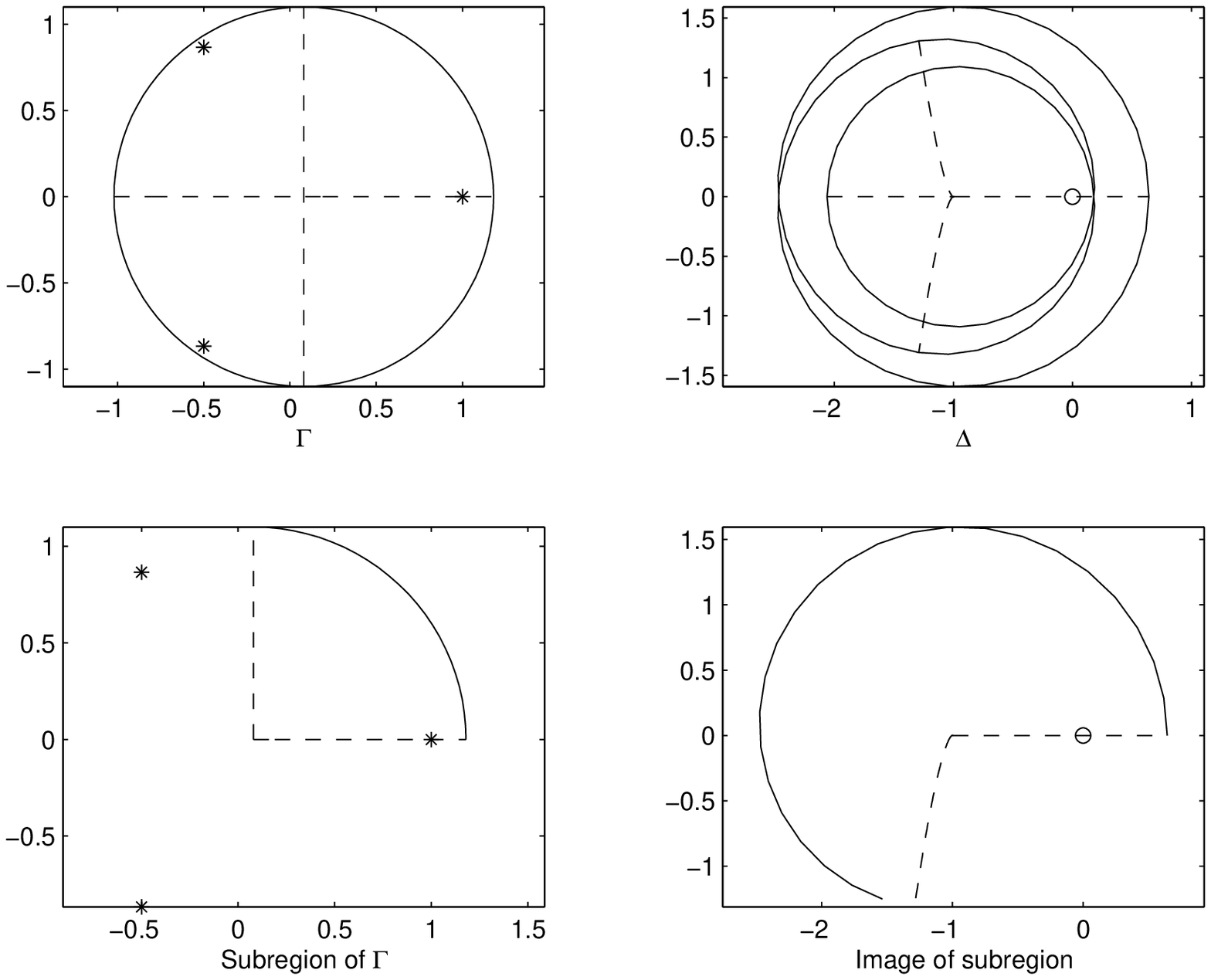} & 
\includegraphics[width=.50\textwidth,trim = .65\textwidth mm .54\textwidth mm 0mm 0mm,clip=true]{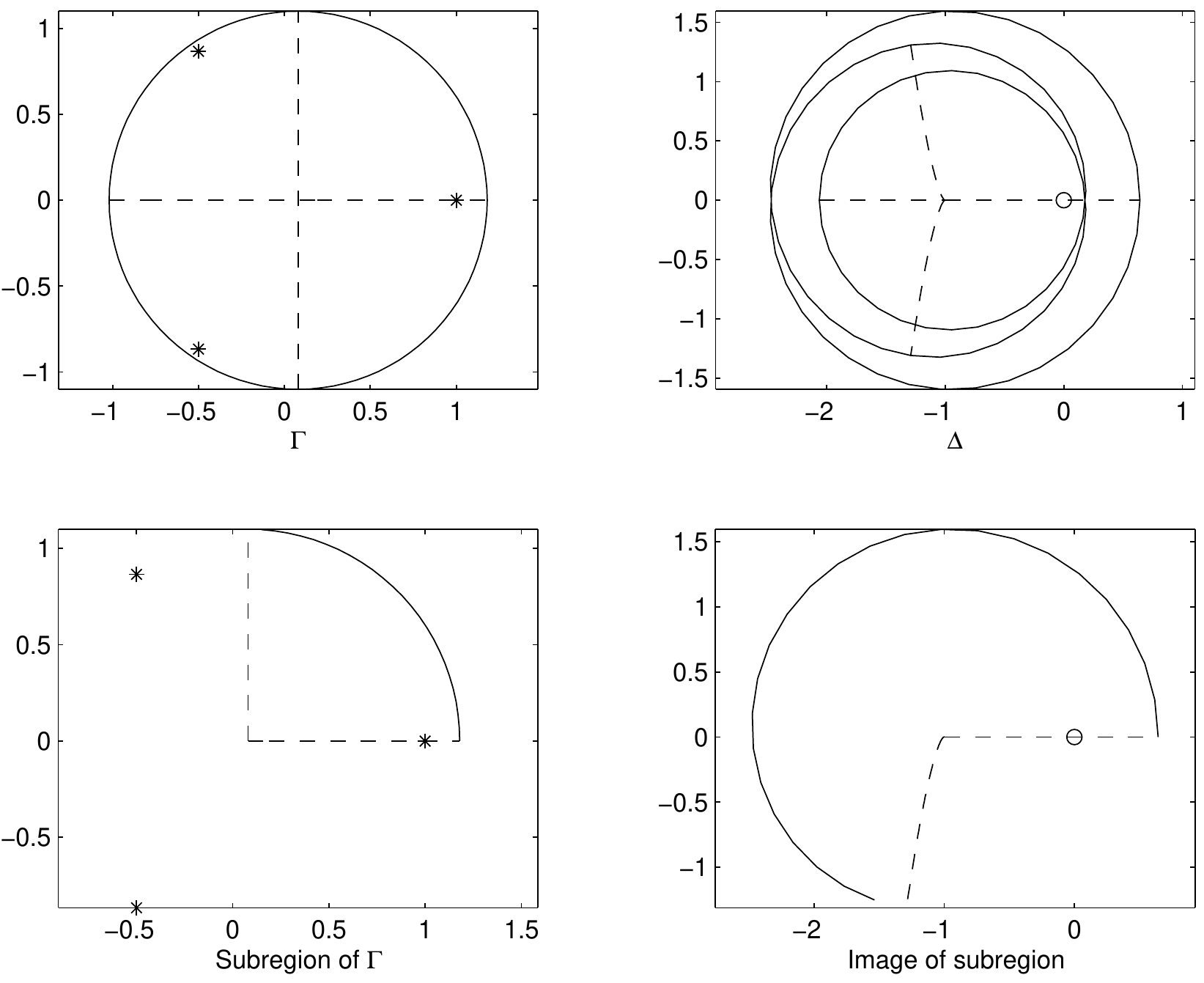}\\
{\em a}) & {\em b}) \\
\includegraphics[width=.50\textwidth,trim = 0mm .01\textwidth mm .65\textwidth mm .47\textwidth mm,clip=true]{images/subreg} & 
\includegraphics[width=.50\textwidth,trim = .65\textwidth mm .01\textwidth mm 0mm .47\textwidth mm,clip=true]{images/subreg}\\
{\em c}) & {\em d})
  \end{tabular}
  \parbox{.9\textwidth}{\caption{Region with the roots of $z^3-1$. The border of the subregion, $\G_s$ crosses over a root.}\label{subreg}}
\end{figure}

The setting of figure \ref{subreg} {\em c}), with a root exactly in the border $\G_s$, can be unlikely, but a an exit with error can also arise can arise if the subregion border has a condition number $\kappa_f(\G_s)$ high enough, by theorem \ref{teorIPSR} {\em c}). This implies, by proposition \ref{distkappa}, that there is a root near $\G_s$.  In short, if the IPSR applied to $f(\G_s)$ with parameter $Q$ returns with error, there is a root at less that $\frac{nQ}{\sin(\pi/8)}$ from $\G_s$, and it does not perform the winding number computation. 

Therefore, the division into four subregions proposed previously is incorrect. To divide in a way that no error is returned by IPSR, we proceed as follows.  Suppose that we start with $\RDP(P,A)$, from an initial region $P$, that does not produce error with IPSR. 
To divide $P$ in subregions also without error for $\IPSR$, let us consider first the subregions that arise from $P$ by an horizontal cut following $m_H(P)$, that is $T(P)$ and $B(P)$. If one of these regions produces an error in IPSR, then the condition number of its border is greater than $\frac{\sin(\pi/8)}{Q}$ by theorem \ref{teorIPSR} {\em c} and there is a point in the border at less than $\frac{nQ}{\sin(\pi/8)}$ from a root by proposition \ref{distkappa}. This point belongs to $m_H(P)$, because the rest of the border of $T(P)$ or $B(P)$ is a border of $P$ too, and this region did not produce an error.

 In such case, with $m_H(P)$ near a root, say $z_1$, we do not divide $P$ using $m_H(P)$, but another line instead that does not give rise to an error exit of IPSR. It is sure that the horizontal line that is at $\frac{2nQ}{\sin(\pi/8)}$ above $m_H(P)$ is at distance greater than $\frac{nQ}{\sin(\pi/8)}$ from $z_1$. Perhaps this second line causes again an error in the IPSR, because there are another root $z_2$ near it. In such case the division is performed following a third horizontal line below $m_H(P)$, at distance $\frac{2nQ}{\sin(\pi/8)}$ (figure \ref{trial}). 

\begin{figure}[h!]
  \centering
  \includegraphics[width=.50\textwidth,trim = 0mm 0mm 0mm 0mm,clip=true]{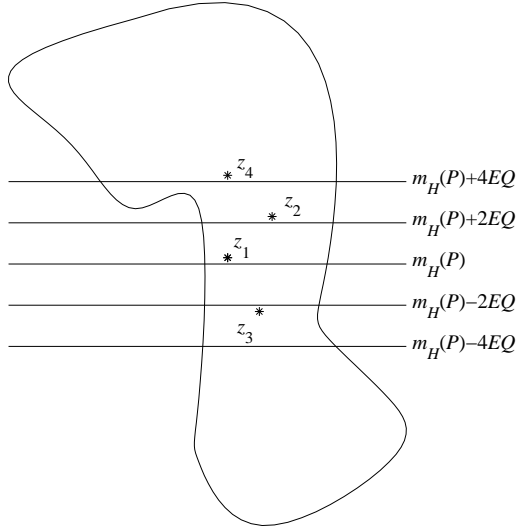} 
  \parbox{.9\textwidth}{\caption{Being $E=\frac{n}{\sin(\pi/8)}$, the lines $m_H(P)$, $m_H(P)+2EQ$, $m_H(P)-2EQ$, $m_H(P)+4EQ$ successively cause error by proximity of the roots $z_1$, $z_2$, $z_3$ and $z_4$ respectively. Finally $m_H(P)-4EQ$ does not have a root at distance lesser than $\frac{nQ}{\sin(\pi/8)}$.}\label{trial}}
\end{figure}

This process, moving the horizontal line a distance of $\frac{2nQ}{\sin(\pi/8)}$ to avoid the roots, will be performed as much $n_0$ times, being $n_0$ the number of roots inside $P$. The half of the $n_0$ or less trial lines are above $m_H(P)$ and the other half below it. As they are $\frac{2nQ}{\sin(\pi/8)}$ apart, the final horizontal cutting line is at less than $n_0\frac{nQ}{\sin(\pi/8)}$ from $m_H(P)$, and does not give place to an IPSR error. A similar reasoning is valid for vertical cutting lines. The division with shifted lines of $P$ is done first with an horizontal cut without error, an then each of the two obtained parts is cut vertically without error. 

We have described the division of the initial region $P$, supposing that IPSR(P,Q) does not return with error. This means that there are no roots near the border of the region of interest, which is a reasonable assumption for $P$. In any case, IPSR must be applied to $P$  before RPD to ensure that it does not fail.

This method of division with shifted lines is recursively applied to the subregions. In the initial region, the assertion that it does not produce IPSR error must be ensured previously, as commented, but for the subregions the equivalent assertion is obtained as a consequence of the cuts shifting.

Remember that the lemma \ref{decrease} gives us a  reduction in diameter for divisions along $m_H(P)$ and $m_V(P)$. The division by the other shifted lines  also produces a decreasing in the diameter, excluding an additive term. We develop a general result about diameters of iterated shifted cuts (in lemma \ref{decreasemu}). Applying this lemma to our way of avoiding the roots (figure \ref{trial}) we can give a value of $Q$ (in proposition \ref{Qbound}) such that the subdivisions are decreasing and besides the calls to $\IPSR$ in this subregions do not return with error.

We define the operator $T_\la$ in the following way. The horizontal line at distance $\la$ above $m_H(P)$  is denoted $m_{H+\la}(P)$. $T_\la(P)$ is the intersection of $P$ with the upper half-plane defined by $m_{H+\la}(P)$. If $\la$ is a negative value, the line $m_{H+\la}(P)$ should be understood below $m_H(P)$. Analogously the operator $B_\la(P)$ is the intersection of $P$ with the lower half-plane defined by the same horizontal line, $m_{H+\la}(P)$. In this way $T_\la(P)\cup B_\la(P)=P$ and $T_\la(P)\cap B_\la(P)$ is its common border, a segment of $m_{H+\la}(P)$. Likewise $R_\la(P)$ is the intersection of $P$ with the right half-plane defined by $m_{V+\la}(P)$ (the vertical line at distance $\la$ at right of $m_V(P)$), and $L_\la(P)$ the intersection with the left half-plane defined by the same line $m_{V+\la}(P)$.  See figure \ref{lambdaops}.

\begin{figure}[h!]
  \centering
  \includegraphics[width=.50\textwidth,trim = 0mm -9.8mm 0mm 0mm,clip=true]{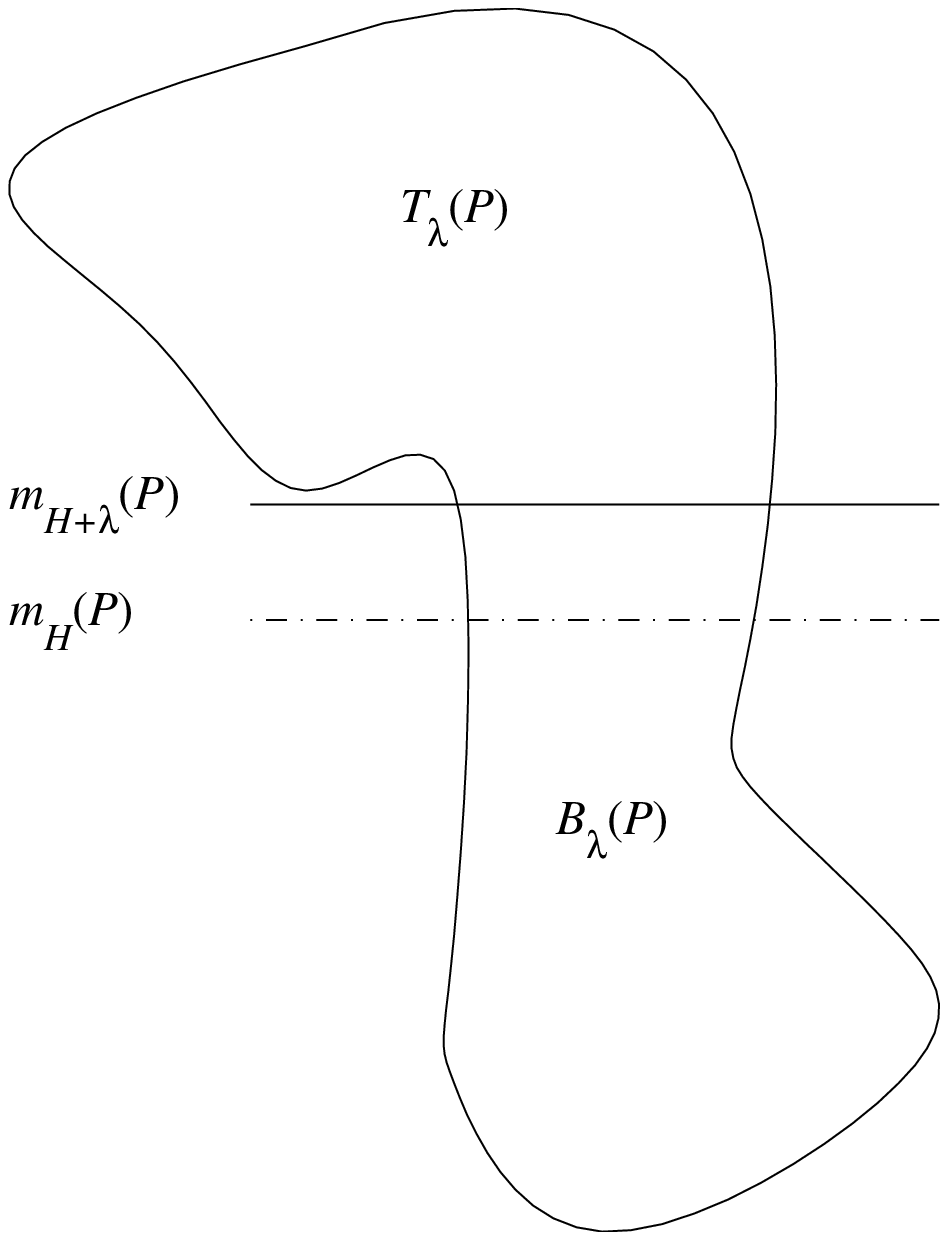}\includegraphics[width=.50\textwidth,trim = 0mm 0mm 0mm 0mm,clip=true]{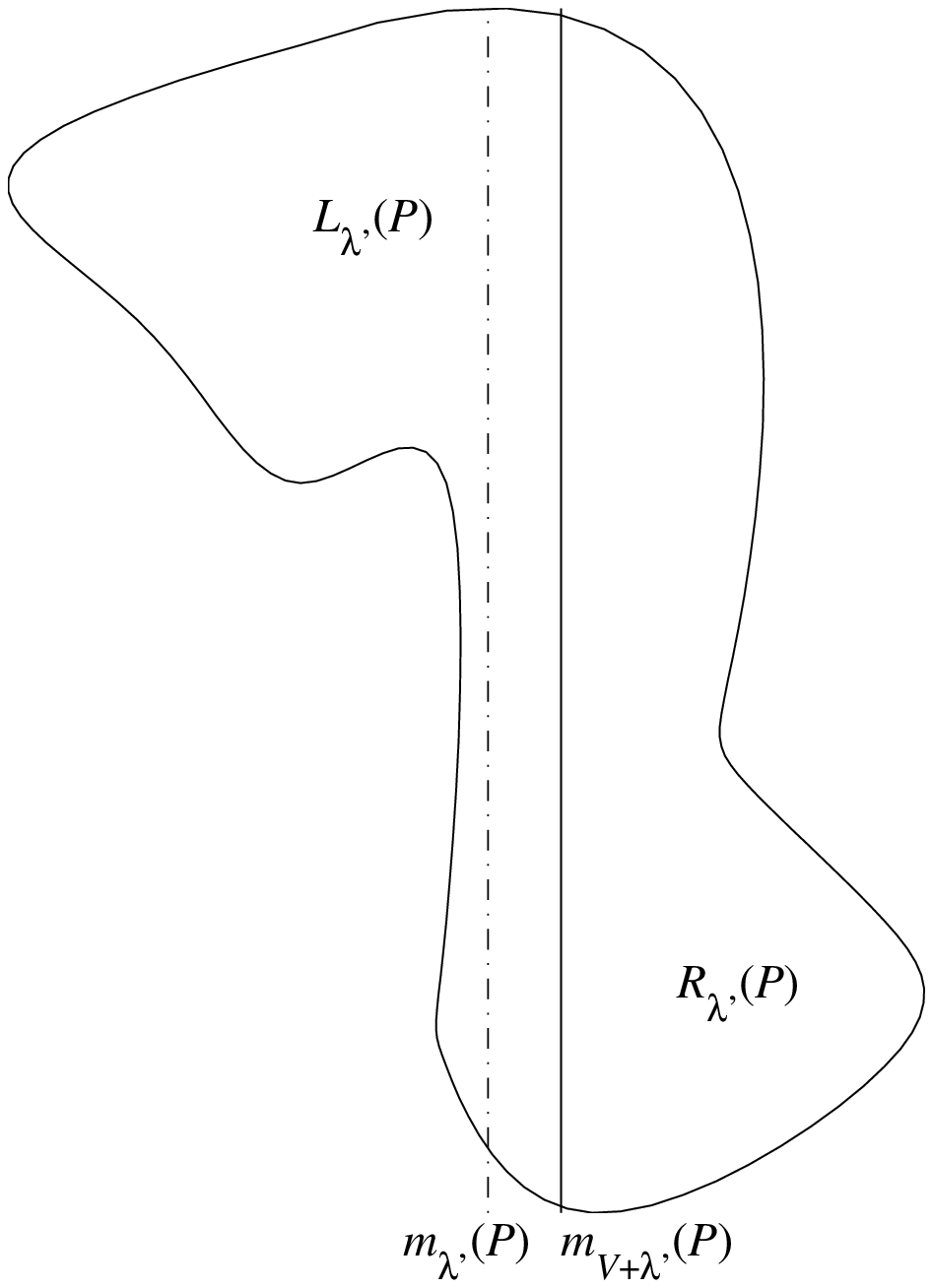} 
  \parbox{.9\textwidth}{\caption{The action of the operators $T_{\lambda}$, $B_\lambda$, $L_{\lambda'}$ and $R_{\la'}$.}\label{lambdaops}}
\end{figure}

\begin{lemman} \label{conten}  If $\la \geq 0$, then $B(P) \subset B_\la(P)$ and  $R(P) \subset R_\la(P)$. If $\la \leq 0$, then  $B(P) \supset B_\la(P)$ and  $R(P) \supset R_\la(P)$. For the operators $T_\la$ and $L_\la$, in the same hypothesis the inclusions are reversed.

\end{lemman}
\begin{proof} It is immediate considering the position of $m_{H+\la}(P)$ and $m_H(P)$ (or of $m_{V+\la}(P)$ and $m_V$(P)).

\qedhere
\end{proof}

We say that the operators $T_\la$ and $B_\la$ are of type {\em horizontal}  and $L_\la$ and $R_\la$ of type {\em vertical}. Let us denote with $H_\la$ a generic horizontal operator, that is, either $T_\la$ or $B_\la$. Similarly $V_\la$ denotes 
either $L_\la$ or $R_\la$. Note that for $\la$ large enough, the result of an operator can be the empty set, whose diameter is conventionally zero.

\begin{lemman} \label{halfmu} For an horizontal operator $H_\la$, $\dm_H(H_\la(P))\leq \dm_H(P)$. If besides $|\la|\leq \dfrac{\dm_V(P)}{2}$ then
$\dm_V(H_\la(P))\leq \dfrac{\dm_V(P)}{2} + |\la|$

For an vertical operator $V_\la$, $\dm_V(V_\la(P))\leq \dm_V(P)$. If besides $|\la|\leq \dfrac{\dm_H(P)}{2}$ then 
$\dm_H(V_\la(P))\leq \dfrac{\dm_H(P)}{2} + |\la|$. 
\end{lemman}
\begin{proof} The first inequality is a consequence of lemma \ref{contain}. 
The condition $|\la|\leq \dfrac{\dm_V(P)}{2}$ makes that the defining line of $H_\la(P)$ will be placed at less than $|\la|$ from $m_H(P)$, above if either $\la\geq 0$ and $H_\la=T_\la$, or $\la\leq 0$ and $H_\la=B_\la$, and below in other cases. Anyway the second inequality is verified. Similarly the claim about vertical operators.

\qedhere
\end{proof}

 An operator $O$ is {\em of type ${\cal O}_{m,\mu}$} for $m\geq 1$ and $\mu\geq 0$ if it is the composition of $m$ operators $H_{\la_i}$ and $m$ $V_{\la'_i}$, $ i= 1,\dots ,m$, alternating, with $|\la_i|\leq\mu,|\la'_i|\leq\mu$, starting with an horizontal operator. That is   $O=V_{\la'_m}H_{\la_m}V_{\la'_{m-1}}H_{\la_{m-1}}  \cdots V_{\la'_1}H_{\la_1}$. The next lemma shows that the diameters of the regions $O(P)$ are lesser than that of $P$, except for an additive term.

\begin{lemman} \label{decreasemu} For $m\geq 1$, $\mu\geq 0$, if $O_m$ is an operator of type ${\cal O}_{m,\mu}$ then
$$
\dr(O_m(P))\leq \dfrac{\dr(P)}{2^m}+\dfrac{2^m-1}{2^{m-1}}\mu\sqrt{2}
$$
\end{lemman}
\begin{proof}
First, note that $\dr(P)=\|(\di_H(P),\di_V(P))\|$, being $\|\,\|$ the Euclidean plane norm. Also, by the triangle inequality $\|(a+\mu,b+\mu)\|\leq \|(a,b)\| + \|(\mu,\mu)\|= \|(a,b)\| + \mu\sqrt{2}$.

We prove the lemma by induction on $m$. For $m=1$, we have that  $O=V_{\la'_1}H_{\la_1}$. Hence, applying some inequalities of lemmas \ref{conten} and \ref{halfmu}: 
\begin{multline*}
\dr(V_{\la'_1}H_{\la_1}(P))= \sqrt{\dm_H(V_{\la'_1}H_{\la_1}(P))^2 +\dm_V(V_{\la'_1}H_{\la_1}(P))^2} \leq \\ 
\leq \sqrt{\left(\dfrac{\dm_H(H_{\la_1}(P))}{2}+|\la'_1|\right)^2
 + \dm_V(H_{\la_1}(P))^2}
  \leq \\
 \leq \sqrt{\left(\dfrac{\dm_H(P)}{2}+|\la'_1|\right)^2 +\left(\dfrac{\dm_V(P)}{2}+|\la_1|\right)^2} \leq \\
\leq  \sqrt{\left(\dfrac{\dm_H(P)}{2}+\mu\right)^2 +\left(\dfrac{\dm_V(P)}{2}+\mu\right)^2} =
\\
=\left\|\left(\dfrac{\dm_H(P)}{2}+\mu, \dfrac{\dm_V(P)}{2}+ \mu\right)\right\| \leq \dfrac{\dr(P)}{2}+ \mu\sqrt{2}
\end{multline*}

For $m>1$, if $O_m$ is of type ${\cal O}_{m,\mu}$, then   $O=V_{\la'_m}H_{\la_m}O_{m-1}$ with  $O_{m-1}$ of type ${\cal O}_{m-1, \mu}$. In the first case (the second are similar), applying again the lemmas and the hypothesis of induction, we have:
\begin{equation*}
\dr(O_m(P))=\dr(V_{\la'_m}H_{\la_m}O_{m-1}(P)) \leq \dfrac{\dr(O_{m-1}(P))}{2}+ \mu\sqrt{2}\\
\end{equation*}

By the hypothesis of induction, $\dr(O_{m-1}(P))\leq \dfrac{\dr(P)}{2^{m-1}}+\dfrac{2^{m-1}-1}{2^{m-2}}\mu\sqrt{2}$, hence
\begin{multline*}
\dfrac{\dr(O_{m-1}(P))}{2}+ \mu\sqrt{2}\leq \dfrac{\dfrac{\dr(P)}{2^{m-1}}+\dfrac{2^{m-1}-1}{2^{m-2}}\mu\sqrt{2}}{2}+ \mu\sqrt{2} =\\
=\dfrac{\dr(P)}{2^{m}}+\dfrac{2^{m-1}-1}{2^{m-1}}\mu\sqrt{2} + \mu\sqrt{2}
=\dfrac{\dr(P)}{2^{m}}+\dfrac{2^{m}-1}{2^{m-1}}\mu\sqrt{2}
\end{multline*}

\qedhere
\end{proof}

Note that $\dfrac{2^{m}-1}{2^{m-1}}=\left(2-\frac{1}{2^{m-1}}\right)<2$, hence $\dr(O_m(P))< \dfrac{\dr(P)}{2^m}+2\sqrt{2}\mu$. 

For any values $\la,\la'_1$ and $\la'_2$ lesser or equal than $\mu$, we say that the   subregions $P_1=R_{\la'_1}T_{\la}(P)$, $P_2=L_{\la'_1}T_{\la}(P)$, $P_3=R_{\la'_2}B_{\la}(P)$ and $P_4=L_{\la'_2}B_{\la}(P)$ are a $\mu$-decomposition of $P$. That is, we first cut horizontally with a $\la$ shift, and then $T_{\la}(P)$ and $B_{\la}(P)$ are cut vertically with respective shifts $\la'_1$ and $\la'_2$. The separation of the cut lines $m_{H+\la}(P),m_{V+\la}(P)$ from the center lines $m_H(P),m_V(P)$ has a maximum value of $\mu$.

The last term of the inequality of lemma \ref{decreasemu}, $2\sqrt{2}\mu$,  obstructs the decreasing of the rectangular diameters of the regions that arise after successive $\mu$-decompositions. After any number of $\mu$-decompositions, the rectangular diameter can be greater than this last term.

We apply this theory of $\mu$-decompositions to our approach to avoiding the roots (figure \ref{trial}),  and show how to get a decreasing in rectangular diameter up the desired precision even considering the obstruction of lemma \ref{decreasemu}. Remember that we want to reach a subregion of diameter lesser than $A$, with cuts shifted a maximum separation of $\mu\leq \frac{n_0nQ}{\sin(\pi/8)}$, being $n_0$ the number of roots inside the original region $P$. We have that:

\begin{propositionn} \label{Qbound} If $Q\leq\frac{A\sin(\pi/8)}{4\sqrt{2}n_0n}$, there is a number $m$ such that any operator $O_m$  of type ${\cal O}_{m,\mu}$ verifies $\dr(O_m(P))\leq A$.
\end{propositionn}
\begin{proof} If we apply to $P$ an operator $O_m$  of type ${\cal O}_{m,\mu}$, by lemma \ref{decreasemu} $$\dr(O_m(P))\leq \dfrac{\dr(P)}{2^m}+2\mu\sqrt{2}
$$

First, note that for some $m$, $\frac{\dr(P)}{2^m}\leq\frac{A}{2}$. Second,  with $Q\leq \frac{A\sin(\pi/8)}{4\sqrt{2}n_0n}$, as $\mu\leq \frac{n_0nQ}{\sin(\pi/8)}$, we have $2\sqrt{2}\mu\leq 2\sqrt{2}\frac{n_0n\frac{A\sin(\pi/8)}{4\sqrt{2}n_0n}}{\sin(\pi/8)}=\frac{A}{2}$. Adding these bounds:
$$\dfrac{\dr(P)}{2^m}+2\mu\sqrt{2} \leq \frac{A}{2}+\frac{A}{2}=A$$
as we wanted to show


\qedhere
\end{proof}
 
\subsection{Termination and cost of the recursive procedure}
\label{termination}

The following theorem assures that the recursive procedure RPD of figure \ref{recur} ends and that its output verifies the claimed property. To make reference to the nested calls of RPD, we say that the first call has {\em recursion level $0$}, and the calls to the same function made inside a call of recursion level $v$ has {\em recursion level $v+1$}.

\begin{Theorem}\label{refinetheor} If $P$ is a plane region containing $n_0$ roots of the polynomial $f,$ with $\dr(P)>A$ and such that IPSN$(P,Q)$ does not produce error, the procedure RPD applied to it with an accuracy of $A$ verifies that:
\begin{itemize}
\item[a)] It ends, reaching a level of recursion lesser or equal than $\ld\left(\frac{\dr(P)}{A}\right)+2$. 
\item[b)] At the end, the plane regions of array $R=(P_1,P_2,\dots,P_k)$ are approximations of all the roots of $f$ in $P,$ each containing the number of roots  given by $N=(n_1,n_2,\dots,n_k)$.
\end{itemize} 
\end{Theorem}
\begin{proof} For {\em a}), by proposition \ref{Qbound}, choosing $Q\leq\frac{A\sin(\pi/8)}{4\sqrt{2}n_0n}$ we have that $2\sqrt{2}\mu\leq\frac{A}{2}$ and, by lemma \ref{decreasemu} $\dr(O_m(P))\leq \dfrac{\dr(P)}{2^m}+\frac{A}{2}$ for any operator  $O_m$  of type ${\cal O}_{m,\mu}$. Note that in the recursion level $v$ the method of decomposition apply operators of type ${\cal O}_{m,\mu}$ with $m=v$. Then, to decrease the rectangular diameter under $A$ is sufficient to reach a recursion level $v$ verifying $\dfrac{\dr(P)}{2^v}+\frac{A}{2}\leq A$, that is $\dfrac{\dr(P)}{2^v}\leq\frac{A}{2}$.  
 
When the  subregion $O_m(P)$ that arise in a call have its diameter below $A$, the subsequent call to RPD returns non recursively by exit 2 (figure \ref{recur}). Therefore we can reach a recursion level $v$ with $\dfrac{\dr(P)}{2^v}\leq\frac{A}{2}$, but the next call will returns by exit 2 (or exit 1 if it is the case) without posterior recursive calls. The maximum recursion level is $m+1$ with $\dfrac{\dr(P)}{2^v}\leq\frac{A}{2}$, that is 
\begin{gather*}
\dfrac{\dr(P)}{2^v}\leq\frac{A}{2} \\
\ld(\dr(P))\leq v+\ld(A) -1\\
\ld(\dr(P))-\ld(A) +1 \leq  v \\
\ld\left(\frac{\dr(P)}{A}\right)+1 \leq  v
\end{gather*}

We have shown that the recursion level reached by RDP is bounded by $\ld\left(\frac{\dr(P)}{A}\right)+2$. As the body of the procedure, apart from recursive calls, ends in a finite number of steps, and the recursion level is bounded, the procedure also ends  in a finite number of steps.

For {\em b}), we prove that at the end of the procedure RDP$(P,A)$, for any plane region $P$, the arrays R and N verify the assertion ``The plane regions of $R$ have diameter lesser than $A$, containing the number of roots given by $N$, and these are all the roots of $f$ contained in $P$''. We will prove this by structural induction (\cite{burstall, Hopcroft}), that is, first in the case of non recursive calls and then, in the other cases, assuming true  the corresponding assertions at the end of the four recursive calls RDP$(P_i,A)$, i.e. assuming true ``The plane regions added to $R$ have diameter lesser than $A$, containing the number of roots given by the number added to $N$, and these are all the roots of $f$ contained in $P_i$''. 

We say that plane region $P$ is {\em of level} $v$ if the procedure  RDP$(P,A)$ reach up to recursion level $v$. Every plane region has a finite level, lesser or equal than $\ld\left(\frac{\dr(P)}{A}\right)+2$ as we see in {\em a}). The structural induction over regions can be viewed as a complete induction (the usual type of induction) over the level $v$ of the regions.

The base case ($v=0$) are the regions that does not require recursive calls, those without roots (that return by exit 1) or with one or more roots but with diameter lesser than $A$, (that return by exit 2). On returning by exit 1 the assertion is trivially verified because $P$ does not have roots and $R$ (and $N$) remains voids. On returning by exit 2 the assertion is also verified because $R$ contains only the region $P,$ and as the IPSR computes without error the winding number, $N$ contain the number of roots inside $P.$

For the general case ($v>1$) four recursive calls RDP$(P_i,A)$ are performed, being $P_1,P_2,P_3$ and $P_4$ disjoint regions that cover $P,$  of level $v-1$ or lesser.  Lets we call $R_i$ the array of regions added to $R$ by RDP$(P_i,A)$. Similarly $N_i$ is the array of numbers added to $N$. 
By induction hypothesis, as $P_i$ is of level $v-1$, at the end of these calls the regions pertaining to $R_i$ are all of diameter lesser than $A$, contains the number of roots given by $N_i$, and these are all the roots in $P_i$. Then the concatenation of $R_i$ and $N_i$ (that is, $R$ and $N$) verify that its regions has diameter and roots as specified. That these are all the roots inside $P$ comes from the fact that the $P_i$ are disjoints and cover $P$.       
\qedhere 
\end{proof}

For the computational cost of the entire root finding method, we count the number of polynomial evaluations that it requires. The
determination of $f(z)$ for a complex $z$, that is, the polynomial evaluation
(PE) is the main operation in the IPSR, in the insertion of a parameter value $s_i$ in array $S$, that requires the computation of the point $f(\G(s_i))$ to evaluate the properties $p$ and $q$ (see figure \ref{IPSR}). The rest of the RDP procedure  consists of computations of $m_H(O(P))$ or $m_V(O(P))$ for operators $O$, in last instance reducible to the mean of maximum and minimum, that is, the traverse of the array of inserted points plus a division. So we consider the cost of these operations negligible compared with the polynomial evaluation of the point themselves.

In general, a polynomial of degree $n$ can be evaluated in a complex
point with $\left\lfloor\frac{n + 1}{2}\right\rfloor$ multiplications using a complex Horner scheme \cite{knuth}.
Each of these complex multiplications can be performed with three float
operations; hence a PE is roughly equivalent to $\frac{3(n+1)}{2}$ floats operations.

\begin{Theorem} The number of PE performed by $\RDP(P,A)$, for a polynomial $f$ of degree $n$ with $n_0$ roots inside $P$, is lesser or equal than $$\frac{8\sqrt{2}n_0n}{A\sin(\pi/8)}  +  (4+\sqrt{2}) \frac{\sqrt{2}n_0n}{\sin(\pi/8)} \left(\ld\left(\frac{\dr(P)}{A}\right)+2\right). $$
\end{Theorem}
\begin{proof} 
We add the number of PE needed by the application of IPSR in the initial region $P$, with those needed by the subregions arising in the recursive calls up to a maximum level of $\ld\left(\frac{\dr(P)}{A}\right)+2$. By proposition \ref{Qbound} the values of $Q$ verifying $Q\leq\frac{A\sin(\pi/8)}{4\sqrt{2}n_0n}$ make RDP ends. By theorem \ref{teorIPSR} {\em a}), the number of iterations of IPSR, that is, of PE, is lesser than $\left\lfloor\frac{b-a}{Q}-1\right\rfloor$. The best $Q$ to use in IPSR is  the greater value of $Q$ that assure us the termination, that is $Q_0=\frac{A\sin(\pi/8)}{4\sqrt{2}n_0n}$. With this value, the number of PE required by IPSR in the initial region is lesser than $\left\lfloor\frac{b-a}{Q}-1\right\rfloor\leq\frac{b-a}{Q_0}$, that is, $\frac{4\sqrt{2}n_0n(b-a)}{A\sin(\pi/8)}$.

To see how many PE requires the division of a region in four parts, that is,  each cut by an horizontal and vertical line, let us consider first the region $T_\la(P)$.  Calling $\D$ to the border of $P,$ the number of PE required by $\D$ has been considered in the above paragraph. For  the border $\D_T$ of $T_\la(P)$, it is build concatenating a curve segment of $\D$ (its upper half) and a straight segment of the cutting line $m_H(H)$. We parameterize uniformly this straight segment, that has length lesser than $\dm_H(P)$.
The number of  PE required by IPSR in $\D_T$ come only from the insertion of points in this last segment, because the insertions in the upper half of $\D$ were previously made.   The number of PE required to insert points at distance lesser than $Q_0$ in a segment of length $\dm_H(P)$ is $\frac{\dm_H(P)}{Q_0}$, that is equal to $\frac{4\sqrt{2}n_0n}{A\sin(\pi/8)} \dm_H(P)$. After the insertion of these points we can apply the IPSR to $T_\la(P)$.
The region $B_\la(P)$ does not require more PE than $T_\la(P)$, because its border is composed from the same straight segment (followed in inverse direction) and the lower half of $\D$. 
  
Subsequently, the regions $T_\la(P)$ and $B_\la(P)$ are themselves divide by vertical cuts. Following the same reasoning, the sum of the number of PE required by the two cuts is lesser than $\frac{4\sqrt{2}n_0n}{A\sin(\pi/8)}\big(\dm_V(T_\la(P))+\dm_V(B_\la(P))\big)$, that is lesser or equal than $$ \frac{4\sqrt{2}n_0n}{A\sin(\pi/8)}\left(\dfrac{\dm_V(P)}{2}+\mu+\dfrac{\dm_V(P)}{2}+\mu\right)$$ by lemma \ref{halfmu} and using $|\la| \leq\mu$. 
Hence, to construct the four parts (and compute the number of roots inside) when we apply RDP to $P$ demands a number of PE less than $\frac{4\sqrt{2}n_0n}{A\sin(\pi/8)} \big(\dm_H(P)+\dm_V(P) + 2\mu \big)$.

After this discussion of the division of the initial region, let us consider the general case. A call to RDP of recursion level $m$ has a entry parameter a region of the form $O_m(P)$ with $O_m\in{\cal O}_{m,\mu}$, and hence the division in four subregions require $\frac{4\sqrt{2}n_0n}{A\sin(\pi/8)}\big(\dm_H(O_m(P))+\dm_V(O_m(P)) + 2\mu \big)$ PE, similarly to the initial region.  By lemma \ref{decreasemu} this is lesser or equal than 

\begin{gather*}
\frac{4\sqrt{2}n_0n}{A\sin(\pi/8)} \left(\dfrac{\dr(P)}{2^m}+2\sqrt{2}\mu +\dfrac{\dr(P)}{2^m}+2\sqrt{2}\mu+2\mu  \right)= \\
=\frac{4\sqrt{2}n_0n}{A\sin(\pi/8)}\left(\dfrac{\dr(P)}{2^{m-1}}+(4\sqrt{2}+2)\mu  \right)\leq \\
\leq\frac{4\sqrt{2}n_0n}{A\sin(\pi/8)}\left(\dfrac{\dr(P)}{2^{m-1}}+(1+\frac{\sqrt{2}}{4})A  \right)
\end{gather*}

The last inequality comes from $\mu\leq \frac{A}{4\sqrt{2}}$ .

By theorem \ref{refinetheor} {\em a}), the maximum recursion level  is  $l_{max}=\ld\left(\frac{\dr(P)}{A}\right)+2$, and if the initial region $P$ have $n_0$ roots at most $n_0$ branches reach up to this level. Hence adding the above bounds in the number of PE in each level we have that the total number is lesser or equal than
\begin{gather*}
\sum_{m=1}^{l_{max}} \frac{4\sqrt{2}n_0n}{A\sin(\pi/8)}\left(\dfrac{\dr(P)}{2^{m-1}}+(1+\frac{\sqrt{2}}{4})A  \right)=
\\
= \frac{4\sqrt{2}n_0n}{A\sin(\pi/8)} \left(\sum_{m=1}^{l_{max}}\dfrac{\dr(P)}{2^{m-1}} + l_{max}(1+\frac{\sqrt{2}}{4})A 
\right)<
\\
<\frac{4\sqrt{2}n_0n}{A\sin(\pi/8)} \left(2\dr(P)+ l_{max}(1+\frac{\sqrt{2}}{4})A 
\right)=
\\
=\frac{8\sqrt{2}n_0n}{A\sin(\pi/8)}  +  (4+\sqrt{2}) \frac{\sqrt{2}n_0nl_{max}}{\sin(\pi/8)} 
\end{gather*}
as we want to show.

\qedhere
\end{proof}

Taking values in the expression, this bound is lesser than $$\frac{30n_0n}{A}  +  21 n_0n \left( \ld\left(\frac{\dr(P)}{A}\right)+2\right).$$

For a given degree $n$, this bound on the number of PE is of order $O\left(\frac{1}{A}\right)$. 

\section{Conclusions}

We have develop an algorithm for root finding of polynomials. It allows to restrict the search to an area of interest in the complex plane, unlike the more common iterative algorithms.

The algorithm combines an inclusion test (to see if a plane region contains some roots) based in the winding number, with a recursive method of partition in subregions, like others known algorithms \cite{ying,renegar,schonhage,pan,neff-reif}. The contribution is that our method to compute the winding number of curves \cite{garciaart1} detects singular cases (that is, when the curve in question cross over some root). This allows to build sub-regions avoiding these cases, in contrast to other algorithms. Therefore we can also give a formal proof of the correctness of the method.   

About the computational complexity, the bound on the number of polynomial evaluations required by the algorithm is of order $O\left(\frac{1}{A}\right)$ with respect the precision $A$. Although our interest in this development is mainly practical, this complexity can be compared with other theoretic bounds, being better than \cite{Pan01} and similar to \cite{neff-reif}. In terms of complexity with respect to the degree, which is $O(n_0n)$, is lesser or equal than quadratic, being $O(n\lg n)$ that of \cite{neff-reif}. 

In practice the cost of our algorithm can be far lower, because the hypothesis in subsection \ref{termination} that every root requires maximal level of recursion can be relaxed. Besides, in the geometric root finding methods, is usual to implement a iterative
procedure (of faster convergence rate, like Newton's) to search a root when
the diameter of the region are below certain bound which assure this convergence (for example the bound of \cite{traub}). This feature is not commented in
this work, but it lowers the cost without compromising the correctness. 

Another way of decrease the expected number of polynomial evaluations is to start the procedure with loose estimates of the parameters involved. In case that this produces an error, the calculation must be rerun using the theoretical values assuring us correctness. Related to this, the fact that the bound of the cost involves  a global property of the polynomial (like  the Lipschitz constant of the polynomial) has been viewed as a handicap of geometrical algorithms \cite{kravanja}. The same technique that with the parameters of the procedure can be used with the value of the global properties: start with a loose estimation and strength it in case of error.

\bibliographystyle{plain}
\bibliography{references}

\end{document}